\documentclass[dvips,preprint]{imsart}
\usepackage{amsfonts}
\usepackage{graphicx}
\usepackage{dsfont}
\usepackage{multicol}
\usepackage[english]{babel}
\usepackage[latin1]{inputenc}
\usepackage{url}
\usepackage{verbatim}
\usepackage{multirow}
\usepackage{fancyhdr}
\usepackage{lastpage}
\usepackage{latexsym}
\usepackage{amssymb}
\usepackage{longtable}
\usepackage{array}
\usepackage[dvips]{epsfig}
\usepackage{lscape}
\usepackage{amsmath}
\usepackage{epsf,epsfig,graphics,subfigure}
\usepackage{dsfont}

\def\Dy#1{\Frac{\partial #1}{\partial y}}

\def\Dy_1y_1#1{\Frac{\partial^2 #1}{\partial y_1^2}}

\def\reff#1{{\rm(\ref{#1})}}

\def\be{\begin{eqnarray}}
\def\ee{\end{eqnarray}}

\def\b*{\begin{eqnarray*}}
\def\e*{\end{eqnarray*}}

\newtheorem{Theorem}{Theorem}[part]

\newtheorem{Proposition}{Proposition}[part]

\newtheorem{Lemma}{Lemma}[part]

\newtheorem{Remark}{Remark}[part]

\makeatletter \@addtoreset{equation}{section}

\@addtoreset{Definition}{section}

\@addtoreset{Theorem}{section}

\@addtoreset{Proposition}{section}

\@addtoreset{Assumption}{section}

\@addtoreset{Corollary}{section}

\@addtoreset{Lemma}{section}

\@addtoreset{Remark}{section}

\@addtoreset{Example}{section}

\makeatother \makeatletter

\def \Sum{\displaystyle\sum}

\def \Frac{\displaystyle\frac}

\def \Sup{\displaystyle\sup}
\def \Lim{\displaystyle\lim}

\def \be{\begin{eqnarray}}
\def \ee{\end{eqnarray}}
\def \b*{\begin{eqnarray*}}
\def \e*{\end{eqnarray*}}

\def \E{\mathbb{E}}

\def \R{\mathbb{R}}

\def \[{[\,\!\![}
\def \]{]\,\!\!]}

\def \1{{\bf 1}}

\def \ep{\hbox{ }\hfill$\Box$}

\def\reff#1{{\rm(\ref{#1})}}

\def\Bc{{\cal B}}

\def\Fc{{\cal F}}

\newtheorem{thm}{Theorem}[section]

\newtheorem{df}[thm]{D\'efinition}

\begin{document}
\begin{frontmatter}
%\title{Numerical scheme for semilinear  Stochastic PDEs via Backward Doubly Stochastic Differential Equations}
\title{Euler time discretization of Backward Doubly  SDEs and Application to Semilinear SPDEs }
\date{}
\runtitle{}

\author{\fnms{Achref}
 \snm{BACHOUCH}\corref{}\ead[label=e1]{achref.bachouch@hotmail.com}}
%\thankstext{T1}{The work of the first author is supported by the chair \textit{risque de cr\'edit}, F\'ed\'eration bancaire Fran\c{c}aise}
\address{Humboldt University of Berlin\\
Institute of Mathematics
\\
Stochastic Group\\
\printead{e1}
}

\author{\fnms{Mohamed Anis}
 \snm{BEN LASMAR}\corref{}\ead[label=e2]{mohamedanis.benlasmar@lamsin.rnu.tn}}
%\thankstext{T1}{The work of the first author is supported by the chair \textit{risque de cr\'edit}, F\'ed\'eration bancaire Fran\c{c}aise}
\address{University of Tunis El Manar\\
 Laboratoire de Mod\'elisation
 Math\'ematique et Num\'erique\\
 dans les Sciences  de l'Ing\'enieur, ENIT\\
 \printead{e2}
 }

\author{\fnms{Anis}
 \snm{MATOUSSI}\corref{}\ead[label=e3]{anis.matoussi@univ-lemans.fr}}
 \thankstext{t3}{Research partly supported by the Chair {\it Financial Risks} of the {\it Risk Foundation} sponsored by Soci\'et\'e G\'en\'erale, the Chair {\it Derivatives of the Future} sponsored by the {F\'ed\'eration Bancaire Fran\c{c}aise}, and the Chair {\it Finance and Sustainable Development} sponsored by EDF and Calyon }
\address{
 Universit\'e du Maine \\
 Institut du Risque et de l'Assurance\\
Laboratoire Manceau de Math\'ematiques\\\printead{e3}
 }

\author{\fnms{Mohamed}
 \snm{MNIF}\corref{}\ead[label=e4]{mohamed.mnif@enit.rnu.tn}}
\thankstext{t4}{This work was partially supported by the research project MATPYL
of the F\'ed\'eration de Math\'ematiques des Pays de la Loire }
\address{
University of Tunis El Manar \\
Laboratoire de Mod\'elisation
 Math\'ematique et Num\'erique\\
 dans les Sciences  de l'Ing\'enieur, ENIT\\\printead{e4}}
% \affiliation{Some University}

\runauthor{A. Bachouch, M.A. Ben Lasmar,  A. Matoussi, M. Mnif}

\vspace{3mm}
\begin{abstract}
This paper investigates a numerical probabilistic method for the solution of some
semilinear stochastic partial differential equations (SPDEs in short).
%Its is known that Pardoux and Peng \cite{par94} related the latter SPDEs to a backward doubly
%stochastic differential equations (BDSDEs in short).\\
The numerical scheme is based on discrete time approximation for solutions of systems
of decoupled forward-backward doubly stochastic differential equations. Under standard assumptions
on the parameters, the convergence and the rate of convergence of the numerical scheme is proven. The proof is based on a
generalization of the result on the path regularity of the backward equation.
\end{abstract}

\vspace{3mm}


%\vspace{3mm}

\noindent {\bf MSC Classification (2000):}
60H10, 65C30.
\end{frontmatter}

\section{Introduction}
\label{introduction:section}
Stochastic partial differential equations combine the features of partial differential equations
and It\^o equations.
Such equations play important roles in many applied fields such as the filtering of partially observable
diffusion processes, genetic population and other areas. For concrete examples, we send the reader to Pardoux \cite{pardoux}, Krylov and Rozovskii \cite{Kry-Rozv}
and Flandoli \cite{Flandoli}. We study the following SPDE for a predictable random
field $u_t\left( x\right)=u\left( t,x\right)$, satisfying:
{
\begin{equation}
\begin{split}
\label{SPDE}   du_t (x) +  \big(Lu_t (x)
 \;   + &  f(t,x,u_t (x),\nabla u_t \sigma  (x)) \big) \, dt  +  g(t,x,u_t(x),\nabla u_t \sigma (x))\cdot \overleftarrow{dB}_t  = 0, \,\
\end{split}
\end{equation}}
over the time interval $[0,T]$, with a given final condition $u_T = \Phi$ and non-linear deterministic  coefficients $f$ and $g$.
$ L u=
\big( Lu_1, \cdots,  Lu_k \big)$ is a second order differential operator and $\sigma$ is the diffusion coefficient.
%\begin{eqnarray*}
%{L}:=\frac{1}{2}\sum_{i,j=1}^{d}(\sigma\sigma^{\star})_{i,j}\frac{\partial^{2}}{\partial x^{i}\partial x^{j}}
% +\sum_{i=1}^{d}b_{i}\frac{\partial}{\partial x^{i}}.
%\end{eqnarray*}
The differential term with $\overleftarrow{dB}_t$
refers to the backward stochastic integral with respect to a $l$-dimensional Brownian motion on
$\big(\Omega, \mathcal{F},\mathbb{P}, (B_t)_{t\geq 0} \big)$. The backward stochastic integral  in the SPDE  is used because we will
 employ the framework  of Backward Doubly Stochastic Differential Equation (BDSDE in short) introduced  first by Pardoux and Peng \cite{par94}.
It gives a probabilistic representation for the classical solution $ u_t (x)$  of the SPDE \eqref{SPDE}
(written in the integral form) in terms of the following class of BDSDE's:
\begin{equation}\label{BDSDE1}
Y_{s}^{t,x}= \Phi(X_{T}^{t,x}) +\int_{s}^{T}f(r,X_{r}^{t,x},Y_{r}^{t,x},Z_{r}^{t,x})dr
+ \int_{s}^{T}g(r,X_{r}^{t,x},Y_{r}^{t,x},Z_{r}^{t,x})\overleftarrow{dB_{r}} - \int_{s}^{T}Z_{r}^{t,x}dW_{r},
\end{equation}
where  $(X_{s}^{t,x})_{ t\leq s\leq T}$ is a diffusion process  starting from $x$ at time $t$ driven by the finite dimensional
Brownian motion $(W_t)_{t\geq 0}$ and  with infinitesimal generator $L$.
%Thus, solving (\ref{SPDE}) or (\ref{BDSDE1}) is essentially the same.
More precisely, under  some regularity assumptions on the final condition $\Phi$ and coefficients $f$ and $g$ ,
they  proved that $u_t (x) =Y_{t}^{t,x}$ and $\nabla u_t  \sigma (x)  =Z_t^{t,x}$, $\forall (t,x)\in [0,T]\times\mathbb{R}^{d}$.
Then, Bally and Matoussi \cite{BMat}  (see also \cite{MS02} ) showed that the same representation remains true in the case when
the final condition (respectively the coefficients  $f$ and $g$) is only measurable  in $x$   (resp.  are  jointly measurable in
 $(t,x) $ and Lispchitz in $u$ and  $\nabla u$). In this paper,   weak Sobolev solution of the equation \eqref{SPDE} was
  considered,  and the approach was based on stochastic flow techniques (see also \cite{K94a,K94b}).
Moreover, their results were generalized in \cite{MS02} to the case of a larger class of SPDE's \eqref{SPDE} driven  by a Kunita-It\^o
 non-linear noise (see \cite{K94a, K94b, Kunitabook} for more details).  In particular, the Kunita-It\^o non-linear noise covers
 a class of  infinite dimensional time-space white-colored noise (see \cite{gyo}, \cite{Rochner}, \cite{Kloeden}).
The explicit resolution of  semi-linear SPDEs is not generally possible, it is then necessary to resort to numerical methods.

The first approach used to solve numerically nonlinear SPDEs is analytic. It is based on time-space discretization  of the SPDEs.
 The discretization in space can be achieved by different methods such as finite differences, finite elements, spectral Galerkin methods.
Most numerical works on SPDEs concentrated on the Euler finite-difference scheme. Gyongi and Nualart  \cite{gyo1} proved that
 these schemes converge, and Gyongy \cite{gyo2} determined the order of convergence.
 Very interesting results were obtained by Gyongy and Krylov \cite{gyo} considering a  symmetric finite difference scheme for a class
  of linear SPDE driven by infinite dimensional Brownian motion. They proved that the approximation error is proportional to
   $\widehat{h}^2$
 where $\widehat{h}$ is the discretization step in space and by the Richardson acceleration method they even got the error
 proportional to $\widehat{h}^4$.
  Walsh \cite{Walsh} investigated schemes based on the finite elements methods. He studied the rate of convergence of these schemes for parabolic SPDEs, including the Forward and Backward Euler and the Crank-Nicholson schemes. He found a substantially similar rate of convergence to those found for finite difference schemes.\\
   The spectral Galerkin approximation was used by Jentzen and Kloeden \cite{Kloeden}.
  They based their method on Taylor expansions derived from the solution of the SPDE, under some regularity conditions.
Lototsky, Mikulevicius and Rozovskii \cite{lot} used the spectral approach for the numerical estimation of the conditional distribution solution of a linear SPDE known as the Zakai equation.
Further developments on spectral methods can be found in Lototsky \cite{lot1}.

The other alternative for resolving numerically SPDEs is the probabilistic approach by using Monte Carlo methods.
These methods are tractable especially when the dimension of the state process is large unlike the finite difference method.
  Furthermore, their parallel nature provides another advantage to the probabilistic approach: each processor of a parallel computer
  can be assigned the task of making a random trial and doing the calculus independently.
  Milstein and Tretyakov \cite{Milstein} solved a linear Stochastic Partial Differential Equation by using the characteristics method
 (the averaging over the characteristic formula). They proposed a numerical scheme based on the Monte Carlo technique.
  Moreover, they constructed Layer methods for
  linear and semilinear SPDEs. Picard \cite{Picard} considered a filtering problem where the observation was a diffusion
   function corrupted by an independent white noise. He estimated the error caused by a discretization of the time interval.
He obtained some approximations of the optimal filter that can be computed with Monte-Carlo methods. Crisan \cite{crisan} studied
a particle approximation for a class of nonlinear stochastic partial differential equations.

Another probabilistic method to solve a semilinear SPDE is based on the associated BDSDE. It requires weaker assumptions on the SPDE's coefficients.
In the deterministic PDE's case i.e. $g\equiv 0$, the numerical approximation of the BSDE has already been
studied in the literature by Bally \cite{B97}, Zhang \cite{zha}, Bouchard and Touzi \cite{bou1}, Gobet, Lemor and Warin\cite{gob} and Bouchard and Elie
\cite{bou2} among others. Zhang \cite{zha} proposed a discrete-time numerical approximation, by step processes,
for a class of decoupled FBSDEs with possible path-dependent terminal values. He proved a $L^2$-type regularity of
the BSDE's solution, the convergence of his scheme and he derived its rate of convergence. Bouchard and Touzi \cite{bou1} suggested a similar numerical scheme for decoupled FBSDEs. The conditional expectations involved in their discretization scheme were computed by using the kernel regression estimation. Therefore, they used the Malliavin approach and the Monte carlo method for its computation. Crisan, Manolarakis and Touzi \cite{cri} proposed an improvement on the Malliavin weights. Gobet, Lemor and Warin in \cite{gob} proposed an explicit numerical scheme.
 In the stochastic PDEs case, i.e. $g\neq 0$, Aman \cite{Aman} and Aboura \cite{Aboura2} treated the particular case when
  $g$ does not depend on the control variable $z$. Aman \cite{Aman} proposed a numerical scheme following
the idea used by Bouchard and Touzi  \cite{bou1} and obtained a convergence of order $h$ of the square of the $L^2$- error ($h$ is the discretization step in time).  Aboura \cite{Aboura2} studied the same numerical scheme under the same kind of hypothesis, but following Gobet et al. \cite{gobet1}. He obtained a convergence of order $h$  in time and used the regression Monte Carlo method to implement his scheme,
 as in  \cite{gobet1}.

%In the general case i.e. when g depends on Z, Aboura \cite{Aboura} studied the discretisation of the BDSDE, and as
%in Zhang's paper, he assumed that $Y$ an $Z$ are real-valued processes. He followed Zhang's approach and proved the $L^2$-regularity of $Z$. Then he obtained an order h for convergence in time but using strong hypothesis; he assumed that $f$ and $g$ are of class $C^{3}$ with bounded derivatives up to order 3 uniformly in time, that $b$ and $\sigma$ are of class $C^{3}$ with bounded derivatives and that $\Phi$ is of class $C^{2}$ with bounded derivatives. Finally, he proposed a numerical scheme as in $\cite{zha}$, which implementation remains discussible.\\
In this paper, we extend the approach of Bouchard-Touzi-Zhang in the general case when $g$  also depends  on the control  variable $z$. We emphasize that this generalization is not obvious because of the  strong  impact of the backward stochastic integral term on the numerical approximation scheme. It is known that in the associated Stochastic PDE's \eqref{SPDE},  the term $ g (u, \nabla u) $ leads to a second order perturbation type which explains the contraction condition assumed on $g$ with respect to the variable $z$ (see \cite{par94}, \cite{NP88}).
This scheme is implicit in $Y$ and explicit in $Z$. The convergence of our time-discretization scheme  is proven and  the rate of convergence given.
The square of the $L^2$- error has an upper bound in the order of  the discretization step in time.
As a consequence, a scheme for the weak solution of the associated  semilinear SPDE is obtained and a rate of convergence result for the later weak solution given. Then, we propose
a fully implementable numerical scheme based on iterative regression functions which are approximated by projections on vector space of functions with coefficients evaluated using Monte Carlo simulations. Finally,  some numerical tests are presented.  Compared to the deterministic numerical method developed by Gyongy and Krylov \cite{gyo},
the probabilistic approach could tackle the semilinear SPDE which could be degenerate and needs fewer regularity conditions on the coefficients than the finite difference scheme.
However, the rate of convergence obtained is clearly slower than the rate obtained by finite difference and finite element schemes,
but our method is available in high dimension.
To simplify the numerical implementation which is based on least-squares method an example is given  in the one dimensional case. For the multidimensional case,
we refer to  Gobet and Lemor \cite{gobet-Lemor} who
 studied the numerical resolution of BSDEs and treated numerical results up to the dimension 10.

The paper is organized as follows: in section 2,  preliminaries and assumptions  are introduced and  the approximation scheme for
 the BDSDEs  \eqref{BDSDE1} is described.  In section 3,  an upper bound result for the time discretization error is shown.
In section 4, we give a Malliavin regularity result for the solution of our Forward-Backward Doubly
SDEs. Then,  we show an  $L^2$-regularity result for the $Z$-component of the solution of the
BDSDEs \eqref{BDSDE1} which is crucial to obtain the rate of convergence
 of our numerical scheme.
Section 5 is devoted to the numerical scheme of the SPDE's weak solution. In section 6,
the convergence of this scheme  is tested statistically by using a path dependent algorithm based on the regression Monte Carlo Method. Finally,  some technical results  are given in the Appendix.

\section{Preliminaries and notations}
\label{preliminary:section}
\setcounter{equation}{0}
\setcounter{Assumption}{0}
\setcounter{Example}{0}
\setcounter{Theorem}{0}
\setcounter{Proposition}{0}
\setcounter{Corollary}{0}
\setcounter{Lemma}{0}
\setcounter{Definition}{0}
\setcounter{Remark}{0}
\subsection{ Forward Backward Doubly Stochastic Differential Equation}
\label{FBDSDE:subsection}
Let $\{W_{s}, 0\leq s\leq T\}$
and $\{B_{s}, 0\leq s\leq T\}$ be two mutually independent standard Brownian motion processes, with
values respectively in $\mathbb{R}^{d}$ and in $\mathbb{R}^{l}$ where $T>0$ is a fixed horizon time, defined on the probability space $(\Omega,\mathcal{F},P)$.\\
We shall work in the product space $\Omega:=\Omega_W\times \Omega_B$, where $\Omega_W$ is the set of continuous functions
from $[0,T]$ into $\mathbb{R}^{d}$ and  $\Omega_B$ is the set of continuous functions
from $[0,T]$ into $\mathbb{R}^{l}$.\\
For $t\in[0,T]$ and $s\in [t, T]$,
\begin{eqnarray*}
\mathcal{F}_{s}^{t}:= \mathcal{F}_{t, s}^{W}\vee \mathcal{F}_{s,T}^{B}
\end{eqnarray*}
is defined, where $\mathcal{F}_{t,s}^{W}= \sigma\{W_{r}-W_{t}, t\leq r\leq s\},
\textrm{ and } \mathcal{F}_{s,T}^{B}= \sigma\{B_{r}-B_{s}, s\leq r\leq T \}$.
 %\mathcal{F}_{t}^{\eta}=\mathcal{F}_{0,t}^{\eta}$.
We set $\mathcal{F}^W=\mathcal{F}_{0,T}^{W}$, $\mathcal{F}^B=\mathcal{F}_{0,T}^{B} $ and $\mathcal{F}= \mathcal{F}^{W}\vee \mathcal{F}^{B}$.\\
We define the probability measures $P_W$ on $(\Omega_W,\mathcal{F}^{W})$ and $P_B$ on $(\Omega_B,\mathcal{F}^{B})$.
Then, we define the probability measure $P:=P_W\otimes P_B$ on $(\Omega,\mathcal{F}^W\times \mathcal{F}^B)$.
Without loss of generality,  it is assumed that $\mathcal{F}^W$ and  $\mathcal{F}^B$ are complete.\\
Note that the collection $\{\mathcal{F}_{s}^{t}, s\in [t,T]\}$ is neither increasing nor decreasing, and it does not constitute a filtration.
To alleviate  notations, we denote $\mathcal{F}_s:=\mathcal{F}_{s}^{0}$.\\

The following spaces are introduced:\\
$\bullet$ $C_{b}^{k}(\mathbb{R}^{p},\mathbb{R}^{q})$ (respectively $C_{b}^{\infty}(\mathbb{R}^{p},\mathbb{R}^{q})$) denotes the set
of functions of class $C^k$ from $\mathbb{R}^{p}$ to $\R^{q}$ whose partial derivatives of order
less or equal to $k$ are bounded (respectively the set of functions of class $C^\infty$ from $\mathbb{R}^{p}$ to $\R^{q}$ whose partial derivatives are bounded).\\
$\bullet$ $C_{b}^{k}([0,T] \times \mathbb{R}^{p},\mathbb{R}^{q})$ denotes the set of functions of class $C^k$ from $[0,T]  \times \mathbb{R}^{p}$ to $\mathbb{R}^q$ whose partial derivatives of order less or equal to $k$ are bounded.\\
$\bullet$ $L^2(\Omega,\mathcal{F}_T,P;\mathbb{R}^k)$ denotes the set of $\mathcal{F}_T$-measurable square integrable
random variables with values in $\mathbb{R}^k$.\\
For any $m\in\mathbb{N}$ and $t\in [0,T] $,  the following notations are introduced:\\
$\bullet$ $\mathbb{H}^{2}_m([t,T])$
denotes the set of (classes of $dP\times dt$ a.e. equal) $\mathbb{R}^{m}$-valued jointly measurable processes $\{\psi_{u}; u\in [t,T]\}$ satisfying:\\
(i) $||\psi||^2_{\mathbb{H}^{2}_m([t, T])}:=E[\int_{t}^{T}|\psi_{u}|^{2}du]<\infty$,\\
(ii) $\psi_{u} $ is $ \mathcal{F}_{u}$-measurable, for a.e. $u\in [t,T]$.\\
$\bullet$ $\mathbb{S}^{2}_m([t, T])$ denotes similarly the set of $\mathbb{R}^{m}$-valued continuous processes satisfying:\\
(i) $||\psi||^2_{\mathbb{S}^{2}_m([t, T])}:=E[\sup_{t\leq u\leq T}|\psi_{u}|^{2}]<\infty$,\\
(ii) $\psi_{u} $ is $ \mathcal{F}_{u}$-measurable, for any $u\in [t,T]$.\\
$\bullet$  $\mathbb{S}$ the set of random variables $F$ of the form: $F=\hat f(W(h_{1}),\ldots,W(h_{m_1}),B(k_{1}),\ldots,B(k_{m_{2}}))$ \\
with $\hat f\in C_{b}^{\infty}(\mathbb{R}^{m_1+m_2},\R)$,
$h_{1},\ldots,h_{m_1}\in L^{2}([t,T],\mathbb{R}^{d}),k_{1},\ldots,k_{m_2}\in L^{2}([t,T],\mathbb{R}^{l})$, where
\begin{eqnarray*}
W(h_{i}):=\int_{t}^{T}h_{i}(s)dW_{s},\quad B(k_{j}):=\int_{t}^{T}k_{j}(s)\overleftarrow{dB_{s}}.
\end{eqnarray*}
For any random variable $F \in \mathbb{S}$,  its Malliavin derivative $(D_{s}F)_s$ is defined  with respect to the Brownian motion
 $W$ as follows
\begin{eqnarray*}
 D_{s}F := \sum_{i=1}^{m_1}\nabla_{i} \hat f \bigg(W(h_{1}),\ldots,W(h_{m_1});B(k_{1}),\ldots,B(k_{m_{2}})\bigg)h_{i}(s),
\end{eqnarray*}
where $\nabla_{i}\hat f$ is the derivative of $\hat f$ with respect to its i-th argument.\\
We define a norm on $\mathbb{S}$  by:
%Then, we define
\begin{eqnarray*}
\|F\|_{1,2}:=\big\{E[F^{2}]+E\big[\int_{t}^{T}|D_{s}F|^{2}ds\big] \big\}^{\frac{1}{2}}.
\end{eqnarray*}
$\bullet$ $\mathbb{D}^{1,2}\triangleq \overline{{\mathbb S}}^{\|.\|_{1,2}}$ is then a Sobolev space.\\
$\bullet$ $\mathcal{S}_k^2([t,T],\mathbb{D}^{1,2})$ is the set of processes $Y=(Y_u,t\leq u\leq T)$
such that $Y\in \mathbb{S}_k^2([t,T])$, $Y^i_u\in \mathbb{D}^{1,2}$, $1\leq i\leq k$, $t\leq u\leq T$ and
$$\|Y\|_{1,2}:=\{E[\int_{t}^{T}|Y_u|^2du]+E[\int_{t}^{T}\int_{t}^{T}||D_{\theta}Y_u||^2du d\theta]\}^{\frac{1}{2}}<\infty.$$
$\bullet$ $\mathcal{M}_{k\times d}^2([t,T],\mathbb{D}^{1,2})$ is the set of processes $Z=(Z_u,t\leq u\leq T)$
such that $Z\in \mathbb{H}_{k\times d}^2([t,T])$, $Z^{i,j}_u\in \mathbb{D}^{1,2}$,$1\leq i\leq k$, $1\leq j\leq d$, $t\leq u\leq T$ and
$$ \|Z\|_{1,2}:=\{E[\int_{t}^{T}\|Z_u\|^2du]+E[\int_{t}^{T}\int_{t}^{T}||D_{\theta}Z_u||^2du d\theta]\}^{\frac{1}{2}}<\infty.$$
$\bullet$ $\Bc^2([t,T],\mathbb{D}^{1,2}):={\cal S}_k^2([t,T],\mathbb{D}^{1,2})\times {\cal M}_{k\times d}^2([t,T],\mathbb{D}^{1,2})$.\\
We define also for a given $t \in [0,T]$:\\
$\bullet$  $L^2([t,T],\mathbb{D}^{1,2})$ is the set of $(\mathcal{F}_s^t)_{s\leq T}$-measurable processes $(v_s)_{t\leq s\leq T}$ such that:\\
(i) $v(s,.)\in \mathbb{D}^{1,2}$, for a.e. $s\in [t,T]$,\\
(ii) $(s,w)\longrightarrow Dv(s,w)\in L^2([t,T]\times \Omega)$,\\
(iii) $E[\int_t^T|v_s|^2ds]+E[\int_t^T\int_t^T|D_uv_s|^2duds]<\infty$.\\
$\bullet$ $L^2([t,T],\mathbb{D}^{1,2}\times \mathbb{D}^{1,2}):=L^2([t,T],\mathbb{D}^{1,2})\times L^2([t,T],\mathbb{D}^{1,2})$.\\

For all $(t,x) \in [0,T]\times\mathbb{R}^{d}$, let $(X_{s}^{t,x})_{0\leq s\leq T}$ be the unique strong solution of the following  stochastic differential equation:
\begin{eqnarray}\label{forward}
dX_{s}^{t,x}=b(X_{s}^{t,x})ds+\sigma(X_{s}^{t,x})dW_{s},\quad s\in [t,T],\qquad X_{s}^{t,x}=x,\quad 0\leq s\leq t,
\end{eqnarray}
where $b$ and $\sigma$ are two measurable functions on $\mathbb{R}^{d}$ with values respectively in $\mathbb{R}^{d}$ and $\mathbb{R}^{d\times d}$. We will omit the dependance of the forward process $X$ in the initial condition if it starts at time $t=0$.\\
We consider the following BDSDE: For all $ t\leq s\leq T$,
\begin{equation}\label{1}
\left\{
\begin{array}{ll}
dY_{s}^{t,x}&= -f(s,X_{s}^{t,x},Y_{s}^{t,x},Z_{s}^{t,x})ds -g(s,X_{s}^{t,x},Y_{s}^{t,x},Z_{s}^{t,x})\overleftarrow{dB_{s}}
+Z_{s}^{t,x} dW_s,\\
Y_{T}^{t,x}&=\Phi(X_{T}^{t,x}),
\end{array}
\right.
\end{equation}
where $f$ and $\Phi$ are two measurable functions respectively on $[0,T]\times\mathbb{R}^{d}\times\mathbb{R}^{k}\times\mathbb{R}^{k\times d}$ and $\mathbb{R}^{d}$
with values in $\mathbb{R}^{k}$ and
 $g$ is a measurable function on $[0,T]\times \mathbb{R}^{d}\times\mathbb{R}^{k}\times\mathbb{R}^{k\times d}$ with values in $\mathbb{R}^{k\times l}$.\\
Note that the integral with respect to $(B_{s},t\leq s\leq T)$ is a "backward It\^o integral" (see Kunita \cite{Kunitabook} and Nualart and Pardoux \cite{NP88} for the definition) and the integral with respect to $(W_{s},t\leq s\leq T)$
is a standard forward It\^o integral.\\
Finally, for each real matrix $A$, we denote by $\|A\|$ its Frobenius norm defined by $\|A\|=(\sum_{i,j}a_{i,j}^2)^{1/2}$.\\
For a vector $x$, $|x|$ stands for its Euclidean norm defined by $|x|=(\sum_i |x_i|^2)^{1/2}$.\\
The following assumptions will be needed in our work:\\

\textbf{Assumption (H1)} There exists a positive constant $K$ such that  $\forall x, x' \in \mathbb{R}^{d}$\\
\begin{eqnarray*}
 &&|b(x)-b(x')| + \|\sigma(x) - \sigma(x')\| \leq K|x-x'|.
\end{eqnarray*}

\textbf{Assumption (H2)} There exist two constants $K>0$ and $0\leq \alpha < 1$ such that\\
for any $(t_1,x_{1},y_{1},z_{1}),(t_2,x_{2},y_{2},z_{2})\in [0,T]\times\mathbb{R}^{d}
\times\mathbb{R}^{k}\times \mathbb{R}^{k\times d} ,$
\begin{eqnarray*}
\textbf{(i)}&&|f(t_1,x_{1},y_{1},z_{1})-f(t_2,x_{2},y_{2},z_{2})|
\leq K \big(\sqrt{|t_{1}-t_{2}|}+|x_{1}-x_{2}|+|y_{1}-y_{2}|+\|z_{1}-z_{2} \|\big),\\
\textbf{(ii)}&&\|g(t_1,x_{1},y_{1},z_{1})-g(t_2,x_{2},y_{2},z_{2})\|^{2}
\leq
K^2\big(|t_{1}-t_{2}|+|x_{1}-x_{2}|^{2}+|y_{1}-y_{2}|^{2} \big)+\alpha^{2} \|z_{1}-z_{2} \|^{2},\\
\textbf{(iii)}&&|\Phi(x_1)-\Phi(x_2)|\leq K|x_{1}-x_{2}|,\\
\textbf{(iv)}&&\Sup_{0\leq t\leq T}(|f(t,0,0,0)|+||g(t,0,0,0)||)\leq K .
\end{eqnarray*}

\textbf{Assumption (H3)}
\begin{eqnarray*}
&\textbf{(i)}& b\in C^{2}_b(\mathbb{R}^{d},\mathbb{R}^{d})\mbox{ and }\sigma\in C^{2}_b(\mathbb{R}^{d},\mathbb{R}^{d\times d})\\
&\textbf{(ii)}&\Phi\in C^{2}_b(\mathbb{R}^{d},\mathbb{R}^{k}), f\in C^{2}_b([0,T]\times\mathbb{R}^{d}\times \mathbb{R}^{k}\times \mathbb{R}^{d\times k},\mathbb{R}^{k})\\
&\mbox{ and }& \quad \quad g\in C^{2}_b([0,T]\times\mathbb{R}^{d}\times \mathbb{R}^{k}\times \mathbb{R}^{d\times k},\mathbb{R}^{k\times l}).
\end{eqnarray*}
We state the following result proved in \cite{par94} (Theorem 1.1. p.212)
\begin{thm}
Under Assumptions \textbf{(H1)} and \textbf{(H2)},  there exists a unique solution
$(Y,Z)$  of the BDSDE \reff{1} which belongs to $ \mathbb{S}^{2}_k([t,T])\times \mathbb{H}^{2}_{k\times d}([t,T])$.
\end{thm}
%Pardoux and Peng \cite{par94} proved that there exists a unique solution
%$(Y,Z)\in \mathbb{S}^{2}_k([t,T])\times \mathbb{H}^{2}_{k\times d}([t,T])$ to the BDSDE \reff{1}.\\
%%Morover,

\begin{Remark}
The regularity conditions on the time-space variable $(t,x)$ of $f$, $g$ and $\Phi$ are needed for the estimates for the time
discretization error of the solution $(Y,Z)$ in section \ref{section-discrete-time-error}.
%Pardoux and Peng \cite{par94} assumed the contraction  condition $0\leq \alpha<1$ to prove the existence and the uniqueness results for the BDSDE's solution.
%% Nevertheless, for the numerical approximation, we need to assume the stronger condition
%%$0 \leq \alpha^2 d<1 $ to make our scheme well-defined (see Remark \ref{remark Picard} ).
\end{Remark}
From \cite{kar}, \cite{par94}  (Theorem 1.4 p. 217) and \cite{kun84}, the standard estimates for the solution of the Forward-Backward Doubly SDE (\ref{forward})-(\ref{1}) hold
and we remind the following theorem:
\begin{thm}
Under Assumptions \textbf{(H1)} and \textbf{(H2)} and for some   $p\geq 2$, there exist  two positive constants $C$ and $C_p$ independent of $x$
 and an integer q such that:
\begin{eqnarray}\label{integrability}
E[\sup_{t\leq s\leq T}|X_{s}^{t,x}|^{2}]\leq C(1+|x|^2),
\end{eqnarray}
%\begin{eqnarray}\label{momentX}
%E[|X_{s}^{t,x}-X_{s'}^{t',x'}|^{p}]\leq C\{|x-x'|^{p}+(1+|x|^p)(|t-t'|^\frac{p}{2}+|s-s'|^\frac{p}{2})\},\,\,\forall s\leq t\mbox { and }s'\leq t'.
%\end{eqnarray}
\begin{equation}\label{apriori1}
E\Big[\sup_{t\leq s\leq T}|Y_{s}^{t,x}|^{p}+\Big(\int_{t}^{T}\|Z_{s}^{t,x}\|^{2}ds \Big)^{p/2} \Big]
\leq C_p(1+|x|^{q}).
\end{equation}%
%For all $s',s\in [t,T], s' \leq s$, we have
%\begin{eqnarray}\label{apriori2yz}
%E\Big[\Sup_{s'\leq u\leq s}|Y_{u}^{t,x}-Y_{s'}^{t,x}|^{2}\Big]
%\leq C\Big((1+|x|^{2})|s-s'|+||Z^{t,x}||_{M^2_{k\times d}[s',s]}\Big).
%\end{eqnarray}
%\begin{equation}\label{apriori2yz,t'}
%\begin{split}
%E\Big[\sup_{t\leq u\leq T}|Y_{s}^{t,x}-Y_{s}^{t',x'}|^{p}+\Big(\int_{t}^{T}\|Z_{s}^{t,x}-Z_{s}^{t',x'} \|^{2}ds \Big)^{\frac{p}{2}}\Big]\\
%\leq C\Big( (1+|x|^{q}+|x'|^{q})(|x-x'|^{p})+|t-t'|^{\frac{p}{2}}\Big) .
%\end{split}
%\end{equation}
\end{thm}
\begin{Remark}
The superscript $(t,x)$ indicates the dependence of the solution $(X,Y,Z)$ on the initial date $(t,x)$.
When it is clear, we omit the dependence of $(Y^{t,x},Z^{t,x})$ on $(t,x)$ .\\
It should also be noted that in the next computations, the constant $C$ denotes a generic constant that may change from line to line.
It depends on K, T, $\alpha$, $|b(0)|$, $||\sigma(0)||$, $|f(t,0,0,0)|$ and $||g(t,0,0,0)||$.
\end{Remark}
\subsection{Numerical Scheme for decoupled Forward-BDSDE}
\label{scheme:subsection}
%%%%%%%%%%%%%%%%%%%%%%%%%%%%%%%%%%%%%%%%%%%%%%%%%%%
%\begin{equation}\label{1}
%- dY_{t} = f(X_{t},Y_{t},Z_{t})dt + g(X_{t},Y_{t},Z_{t})\overleftarrow{dB_{t}} - Z_{t}dW_{t}
%\end{equation}
In order to approximate the solution of the BDSDE (\ref{1}),  the following discretized version is introduced. Let
\begin{eqnarray}\label{disctime}
\pi:t_{0}=0<t_{1}
<\ldots<t_{N}=T,
\end{eqnarray}
be a partition of the time interval $[0,T]$. For simplicity we take an equidistant partition of  $[0,T]$ i.e. $h = \frac{T}{N}$ and
$t_n=n h$, $0\leq n\leq N$.
 Throughout the rest, the notations $\Delta W_{n}=W_{t_{n+1}}-W_{t_{n}}$ and $\Delta B_{n}=B_{t_{n+1}}-B_{t_{n}}$, for $n=1,\ldots,N$ will be used.\\
 The forward component $X$ will be approximated by the classical Euler scheme:
 \begin{eqnarray}
 \label{Xn}
\left\{\begin{array}{ll}
X_{t_{0}}^{N}&=X_{t_{0}},\\
 X_{t_{n}}^{N}&=X_{t_{n-1}}^{N}+b(X_{t_{n-1}}^{N})(t_{n}-t_{n-1})+\sigma(X_{t_{n-1}}^{N})(W_{t_{n}}-W_{t_{n-1}}), \textrm{ for } n=1,\ldots,N.
\end{array}
\right.
\end{eqnarray}
Note the following lemma (see \cite{Kloeden-Platen}):
\begin{Lemma}\label{Lemma-estimation X}
Under Assumption \textbf{(H1)}, there exists a positive constant $C$ independent of $x$ and depending on $K$,$T$, $|b(0)|$ and
$\|\sigma(0)\|$ such that for all $s \in [t_{n},t_{n+1})$ and for all $n=0,\ldots,N-1$ we have:
\begin{eqnarray}\label{estimation X}
E\Big[|X_{s}-X_{t_{n}}^{N}|^{2}+|X_{s}-X_{t_{n+1}}^{N}|^{2}\Big] \leq C h(1+|x|^2).
\end{eqnarray}
\end{Lemma}
%and we set
%\begin{eqnarray}
% X_{t}^{N}:=X_{t_{i-1}}^{N}+b(X_{t_{i-1}}^{N})(t-t_{i-1})+\sigma(X_{t_{i-1}}^{N})(W_{t}-W_{t_{i-1}}), \textrm{ for } t\in(t_{i-1},t_{i}).\nonumber
%\end{eqnarray}
 %We denote by $X^{N}$ the approximation of $X$ at these discretisation times.
%It is known that as $N$ goes to infinity, one has $\Sup_{0\leq n\leq N}E|X_{t_{n}}-X_{t_{n}}^{N}|^{2}\rightarrow 0$.\\
 % $X^{N}$ is assumed to converge to $X$ in $L_{2}$-norm, which is stated as\\
% \textbf{(H3)} As N goes to infinity, one has $\sup_{0\leq k\leq N}E|X_{t_{k}}-X_{t_{k}}^{N}|^{2}\rightarrow 0$\\
The solution $(Y,Z)$ of (\ref{1}) is approximated by $(Y^{N},Z^{N})$ defined by:
\begin{equation}\label{4}
Y_{t_{N}}^{N} = \Phi(X_{T}^{N})\textrm{ and } Z_{t_{N}}^{N}=0,
\end{equation}
and for $n=N-1,\ldots,0$, we set
\begin{equation}\label{Yn}
Y_{t_{n}}^{N} = E_{t_{n}}[Y_{t_{n+1}}^{N}+g(t_{n+1},\Theta_{n+1}^{N})\Delta B_{n}] + h f(t_{n},\Theta_{n}^{N}),
\end{equation}
\begin{equation}\label{Zn}
h Z_{t_{n}}^{N} = E_{t_{n}}\Bigg[Y_{t_{n+1}}^{N}\Delta W_{n}^*
+ g(t_{n+1},\Theta_{n+1}^{N})
\Delta B_{n}\Delta W^*_{n}\Bigg],
\end{equation}
where
\begin{eqnarray*}
\Theta_{n}^{N}:=(X_{t_{n}}^{N},Y_{t_{n}}^{N},Z_{t_{n}}^{N}), \textrm{for all } n=0,\ldots,N.
\end{eqnarray*}
$*$ denotes the transposition operator and
$E_{t_{n}}$ denotes the conditional expectation over the $\sigma$-algebra $\Fc_{t_{n}}$.\\
\begin{Remark}
By construction, $(Y_{t_n}^N,Z_{t_n}^N)_{n\geq 0}$ are square integrable.
For the approximation of $Y_{t_n}^N$, $(\ref{Yn})$ is well-defined, indeed $Y_{t_n}^N(\omega)$ is a fixed point of
\begin{equation*}
\varphi(x)=hf(t_n,X_{t_{n}}^{N}(\omega),x,Z_{t_{n}}^{N}(\omega))+
E_{t_{n}}[Y_{t_{n+1}}^{N}+g(t_{n+1},X_{t_{n+1}}^{N},Y_{t_{n+1}}^{N},Z_{t_{n+1}}^{N})\Delta B_{n}](\omega),
\end{equation*}
which exists and is unique as soon as $Kh<1$. Such a condition holds when $h$ is small enough.
%if we set
%
%$$
%\left\{ \begin{array}{lll}
%Y_{t_{n}}^{N,0} &=& 0 \\
%Y_{t_{n}}^{N,i} &=&  E_{t_{n}}[Y_{t_{n+1}}^{N}+g(t_{n+1},X_{t_{n+1}}^{N},Y_{t_{n+1}}^{N},Z_{t_{n+1}}^{N})\Delta B_{n}] +
% h f(t_{n},X_{t_{n}}^{N},Y_{t_{n}}^{N,i-1},Z_{t_{n}}^{N}) ,\quad i\geq1
%\end{array} \right.
%$$
%
%By using the Isometry property for the It\^o integral, we have
%\b*
%&&E[|Y_{t_n}^{N,i} - Y_{t_n}^{N,i-1}|^2]\\
%&=&h^2E\Big[|E_{t_{n}}\big[f(t_{n},X_{t_{n}}^{N},Y_{t_{n}}^{N,i-1},Z_{t_{n}}^{N})-f(t_{n},X_{t_{n}}^{N},Y_{t_{n}}^{N,i-2},Z_{t_{n}}^{N})\big] |^2\Big]\\
%&\leq& h^2E\Big[|f(t_{n},X_{t_{n}}^{N},Y_{t_{n}}^{N,i-1},Z_{t_{n}}^{N})-f(t_{n},X_{t_{n}}^{N},Y_{t_{n}}^{N,i-2},Z_{t_{n}}^{N})|^2\Big]
%\e*
%By Assumption \textbf{(H2)}, we obtain for $i\geq 2$
%\b*
%E[|Y_{t_n}^{N,i} - Y_{t_n}^{N,i-1}|^2]\leq h^2 K^2 E[|Y_{t_n}^{N,i-1} - Y_{t_n}^{N,i-2}|^2]
%\e*
%This shows that, for $h$ small enough, $(Y_{t_{n}}^{N,i})_{i\geq0}$ is a Cauchy sequence and
%$Y_{t_{n}}^{N,i}$ in $L^{2}(\Omega,\mathcal{F}_{t_{n}}^0)$ where $L^{2}(\Omega,\mathcal{F}_{t_{n}}^0)$
%denotes the set of square integrable random variables which are $\Fc_{t_n}^0$-measurable.
\end{Remark}
For later use, a continuous approximation of the solution of BDSDE (\ref{1}) must be introduced. We define:
\begin{equation}\label{1N}
Y_t^N:=Y_{t_{n+1}}^{N} +  \!\displaystyle{\int_{t}^{t_{n+1}}\!\!\!\!\!\!\!\!f(t_{n},\Theta_{n}^{N})ds}
+\!\displaystyle{\int_{t}^{t_{n+1}}\!\!\!\!\!\!\!\!g(t_{n+1},\Theta_{n+1}^{N})\overleftarrow{dB_{s}}}
-\!\displaystyle{\int_{t}^{t_{n+1}}\!\!\!\!\!\!Z_{s}^{N} dW_s},\textrm{ } t_{n}\leq t < t_{n+1}.
\end{equation}
where
\begin{eqnarray*}
\Theta_{n}^{N}:=(X_{t_{n}}^{N},Y_{t_{n}}^{N},Z_{t_{n}}^{N}), \textrm{for all } n=0,\ldots,N.
\end{eqnarray*}
The following property of $Z^N$ is needed later.
% for the proof of the main result of this paper.
%In the following Lemma, we give a property of $Z^N$ which is needed for the proof of the main result of this paper.
\begin{Lemma}\label{lemme relation Z_set Z_tn}
For all $n=0,\ldots,N-1$, we have
\begin{eqnarray}\label{projectionZN}
Z_{t_{n}}^{N}=\frac{1}{h}E_{t_{n}}[\int_{t_{n}}^{t_{n+1}} Z_{s}^{N}ds]\quad P-a.s.
\end{eqnarray}
\end{Lemma}
\begin{proof}
From (\ref{1N}) we have
\begin{eqnarray*}\label{}
\int_{t_{n}}^{t_{n+1}}Z_{s}^{N}dW_{s}\Delta W_{n}&=&Y_{t_{n+1}}^{N}\Delta W_{n}+
 \int_{t_{n}}^{t_{n+1}} f(t_{n},\Theta_{n}^{N})ds \Delta W_{n}
\nonumber\\
&+&\int_{t_{n}}^{t_{n+1}} g(t_{n+1},\Theta_{n+1}^{N}) \overleftarrow{dB_{s}}\Delta W_{n} - Y_{t_{n}}^{N}\Delta W_{n}.
\end{eqnarray*}
%Then
Taking the conditional expectation we get
\begin{eqnarray*}\label{}
E_{t_{n}}[\int_{t_{n}}^{t_{n+1}} Z_{s}^{N}dW_{s}\Delta W_{n}]&=&E_{t_{n}}[Y_{t_{n+1}}^{N}\Delta W_{n}]+
 \E_{t_{n}}[\int_{t_{n}}^{t_{n+1}} f(t_{n},\Theta_{n}^{N})ds \Delta W_{n}]\nonumber\\
&+& E_{t_{n}}[\int_{t_{n}}^{t_{n+1}} g(t_{n+1},\Theta_{n+1}^{N}) \overleftarrow{dB_{s}}\Delta W_{n}]
-E_{t_{n}}[Y_{t_{n}}^{N}\Delta W_{n}]\nonumber\\
&=&E_{t_{n}}[Y_{t_{n+1}}^{N}\Delta W_{n}]+ h E_{t_{n}}[f(t_{n},\Theta_{n}^{N}) \Delta W_{n}]\nonumber\\
&+&E_{t_{n}}[ g(t_{n+1},\Theta_{n+1}^{N}) \Delta B_{n} \Delta W_{n}]-E_{t_{n}}[Y_{t_{n}}^{N}\Delta W_{n}].
%&=&E_{t_{n}}[Y_{t_{n+1}}^{N}\Delta W_{n}]+E_{t_{n}}[ g(t_{n+1},\Theta_{n+1}^{N}) \Delta B_{n} \Delta W_{n}^*].
\end{eqnarray*}
Using the fact that $Y_{t_{n}}^{N}$ and $f(t_{n},\Theta_{n}^{N})$ are $\mathcal{F}_{t_{n}}$-measurable, we obtain
\begin{eqnarray}\label{relation1}
E_{t_{n}}[\int_{t_{n}}^{t_{n+1}} Z_{s}^{N}dW_{s}\Delta W_{n}]=E_{t_{n}}[Y_{t_{n+1}}^{N}\Delta W_{n}]+E_{t_{n}}[ g(t_{n+1},\Theta_{n+1}^{N}) \Delta B_{n} \Delta W_{n}^*].
\end{eqnarray}
By using the integration by parts formula, we have
\begin{eqnarray*}\label{}
E_{t_{n}}[\int_{t_{n}}^{t_{n+1}} Z_{s}^{N}dW_{s}\Delta W_{n}]&=&E_{t_{n}}[\int_{t_{n}}^{t_{n+1}}\int_{t_{n}}^{s}dW_u Z_{s}^{N}dW_s]\\
&+&E_{t_{n}}[\int_{t_{n}}^{t_{n+1}}\int_{t_{n}}^{s} Z_{u}^{N}dW_udW_s]+E_{t_{n}}[\int_{t_{n}}^{t_{n+1}}Z_{s}^{N}ds].
\end{eqnarray*}
Then
\begin{eqnarray}\label{relation2}
E_{t_{n}}[\int_{t_{n}}^{t_{n+1}} Z_{s}^{N}dW_{s}\Delta W_{n}]=E_{t_{n}}[\int_{t_{n}}^{t_{n+1}}Z_{s}^{N}ds].
\end{eqnarray}
Equations (\ref{relation1}) and (\ref{relation2}) together with  (\ref{Zn}) give that
\begin{eqnarray*}\label{}
h Z_{t_n}^N=E_{t_{n}}[\int_{t_{n}}^{t_{n+1}} Z_{s}^{N}ds].
\end{eqnarray*}
\end{proof}
\ep\\

\section{The discrete time approximation error}\label{section-discrete-time-error}
Fisrt, the step process $\bar{Z}$ is defined by
\begin{eqnarray}\label{barZdef}
\left\{ \begin{array}{lll}
\bar{Z}_{t}=\frac{1}{h}\displaystyle{E_{t_{n}}[\int_{t_{n}}^{t_{n+1}} Z_{s}ds]}, \textrm{ for all } t \in [t_n,t_{n+1}), \textrm{ for all } n \in \{0,\ldots,N-1\},\\
\bar{Z}_{t_{N}}=0.
\end{array} \right.
\end{eqnarray}
The following theorem states an upper bound result regarding the time discretization error.
%We choose the numerical scheme given by (\ref{Yn}) and (\ref{3}). If we choose an explicit numerical scheme, we have the same order of difficulties.
\begin{thm}\label{upper-bound}
%\begin{enumerate}
%\item
Define the square error by
\begin{eqnarray}\label{error(Y,Z)}
Error_{N}(Y,Z):=\sup_{0\leq s \leq T}E[|Y_{s}-Y_{s}^{N}|^{2}] +\sum_{n=0}^{N-1}E[\int_{t_n}^{t_{n+1}}||Z_{s}-Z_{s}^{N}||^{2}ds],
\end{eqnarray}
where $Y^{N}$ and $Z^{N}$ are given by (\ref{1N}).
Under Assumptions \textbf{(H1)} and \textbf{(H2)} we have
\begin{eqnarray}\label{convergence}
Error_{N}(Y,Z) &\leq&  C h(1+|x|^2) + C \sum_{n=0}^{N-1} \int_{t_{n}}^{t_{n+1}}E[||Z_{s}-\bar{Z}_{t_{n}}||^{2}]ds \nonumber\\
&+&  C \sum_{n=0}^{N-1}\! \int_{t_{n}}^{t_{n+1}}\!\!\!\!\!\!\!\!\!\! E[||Z_{s}-\bar{Z}_{t_{n+1}}||^{2}]ds
+ C \sum_{n=0}^{N-1} \int_{t_{n}}^{t_{n+1}}E[|Y_{s}-Y_{t_{n}}|^{2}]ds\nonumber\\
&+& C \sum_{n=0}^{N-1} \int_{t_{n}}^{t_{n+1}}E[|Y_{s}-Y_{t_{n+1}}|^{2}]ds.
\end{eqnarray}
%and $Err_{N}(Y,Z)=O(h)$ .
%\item
%Define the error
%\begin{eqnarray*}
%Error_{N}(u,v):=\max_{0\leq n\leq N}E[\|u_{t_{n}}^{N}-u_{t_{n}}\|_{2}^{2}] +E[\sum_{n=0}^{N-1}
%\int_{t_{n}}^{t_{n+1}}\|v_{t_{n}}^{N}-v_{t}\|_{2}^{2}dt],
%\end{eqnarray*}
%Then, $Error_{N}(u,v)$ converges to 0 as $N\rightarrow\infty$
%\end{enumerate}
\end{thm}
Before proving Theorem \ref{upper-bound}, we need the following lemma whose proof is given in the Appendix.
For all $t\in[t_n,t_{n+1})$, $n=0,\ldots,N-1$, the following quantities are defined:
\begin{equation}\label{}
\left\{
\begin{array}{llll}
\theta_{t}:=(X_{t},Y_{t},Z_{t})\textrm{ },\delta Y_{t}^{N}:=Y_{t}-Y_{t}^{N},\textrm{ }\delta Z_{t}^{N}:=Z_{t}-Z_{t}^{N},\\
\delta f_{t}:=f(t,\theta_{t})-f(t_{n},\Theta_{n}^{N}),\\
\delta g_{t}:=g(t,\theta_{t})-g(t_{n+1},\Theta_{n+1}^{N}).
\end{array}
\right.
\end{equation}
Introduce the following term: for $n\leq N-1$
\begin{equation}\label{}
R_n:= Ch^2(1+|x|^2)+
C\int_{t_n}^{t_{n+1}}E\big[|Y_{s}-{Y}_{t_{n}}|^{2}+|Y_{s}-{Y}_{t_{n+1}}|^{2}+||Z_{s}-\bar{Z}_{t_{n}}||^{2}
+||Z_{s}-\bar{Z}_{t_{n+1}}||^{2}\big]ds
\end{equation}
%%%%%%%%%%%%%%%%%%%%%%%%%%%%%%%%%%%%%%% J'ai changé Lemma 3.1 pour n=0,...,N-1 et donc le fusionner avec Lemma 3.2  %%%%%%%%%%%%%%%%%%%%%%%%%%%%%%%%%%%%%%%%%%%%%%%%%%%%%%%%
\begin{Lemma}\label{first step}
Under Assumptions \textbf{(H1)} and \textbf{(H2)}, there exists a constant $\alpha ' \in (0,1)$
such that for a constant $C>0$
\begin{eqnarray}\label{majYtnZtn}
&&\frac{1}{C}\sup_{t\in[t_n,t_{n+1}]}E[|\delta Y_{t}^{N}|^{2}]+E\Big[|\delta Y_{t_n}^{N}|^{2}+
 \frac{1+\alpha'}{2}\int_{t_{n}}^{t_{n+1}}\!\!\!\!\|\delta Z_{s}^{N}\|^{2}ds\Big]
\leq  \nonumber\\
&&(1+Ch)\Bigg\{E\Big[|\delta Y_{t_{n+1}}^{N}|^{2} + \alpha'\mathds{1}_{\{n<N-1\}} \int_{t_{n+1}}^{t_{n+2}}\!\!\!\!\|\delta Z_s^N\|^2ds\Big]
+ R_n \Bigg\}.
\end{eqnarray}
\end{Lemma}
%\begin{proof}
%See Appendix.
%\end{proof}
%%%%%%%%%%%%%%%%%%%%%%%%%%%%%%%%%%%%%%%% J'ai elliminé Lemma 3.2 ( Step N-1) %%%%%%%%%%%%%%%%%%%%%%%%%%%%%%%%%%%%%%%%%%%%%%%%%%%%%%%%%%%
%\begin{Lemma}\label{step N-1}
%Assume that Assumptions \textbf{(H1)} and \textbf{(H2)} hold, then there exists a constant $\alpha ' \in (0,1)$
%such that we have:
%\begin{eqnarray}\label{majYtnZtn-N-1}
%&&E[|\delta Y_{t_{N-1}}^{N}|^{2}]+ \frac{1+\alpha '}{2}\int_{t_{N-1}}^{t_{N}}\!\!\!\!E[\|\delta Z_{s}^{N}\|^{2}]ds
%\leq  (1+Ch)\Big\{  C h^{2}(1+|x|^2)+ E[|\delta Y_{t_{N}}^{N}|^{2}]\nonumber\\
%&+& \alpha^2 \int_{t_{N-1}}^{t_{N}}\!\!\!\!E[\| Z_s\|^2]ds
%+ C\int_{t_{N-1}}^{t_{N}}E[|Y_{s}-{Y}_{t_{N-1}}|^{2}]ds + C\int_{t_{N-1}}^{t_{N}}E[|Y_{s}-{Y}_{t_{N}}|^{2}]ds \nonumber\\
%&+&  C\int_{t_{N-1}}^{t_{N}}\!\!\!\!\!\!\!E[||Z_{s}-\bar{Z}_{t_{N-1}}||^{2}]ds \Big\}.
%\end{eqnarray}
%\end{Lemma}
%%%%%%%%%%%%%%%%%%%%%%%%%%%%%%%%%%% fin Lemma %%%%%%%%%%%%%%%%%%%%%%%%%%%%%%%%%%%%%%%%ù
%\begin{proof}
%See Appendix.\\
%\end{proof}
%%%%%%%%%%%%%%%%%%%%%%%%%%%%%%%%%%%%%%%%%%%%%%%%%%%%%%%%%%%%%%%%%%%%%%%%%%%%%%%%%%%%%%%%%%%%%%%%%%%%%%%%%%%%%%%%%%%%%
\textbf{Proof of Theorem \ref{upper-bound}.}
To alleviate the presentation, we introduce $y_n:=E[|\delta Y_{t_{n}}^{N}|^{2}]$ and $z_n:=E\Big[\int_{t_{n}}^{t_{n+1}}\|\delta Z_{s}^{N}\|^{2}ds\Big]$.
From Lemma \ref{first step}, we have for all $n=0,..., N-1$
\begin{eqnarray}\label{majYtnZtn-2}
y_n+ \frac{1+\alpha'}{2}z_n
\leq  (1+Ch)\Big( y_{n+1}
+ \alpha' \mathds{1}_{\{n<N-1\}}z_{n+1}+ R_n\Big).
\end{eqnarray}
Summing (\ref{majYtnZtn-2}) from $n=i$ to $n=N-1$, $i\leq N-1$, we obtain
\begin{eqnarray*}\label{}
\sum_{n=i}^{N-1}y_n+\frac{1+\alpha'}{2}\sum_{n=i}^{N-1}z_n\leq
(1+C h)\Big( \sum_{n=i}^{N-1}y_{n+1}+\alpha'\sum_{n=i+1}^{N-1}z_n+
\sum_{n=i}^{N-1} R_n\Big).
\end{eqnarray*}
Then, we have
\begin{eqnarray*}\label{}
y_i+\frac{1+\alpha'}{2}\sum_{n=i}^{N-1}z_n\leq
y_N +C h \sum_{n=i+1}^{N}y_n
+(1+Ch)\alpha'\sum_{n=i+1}^{N-1}z_n
+(1+C h)\sum_{n=i}^{N-1} R_n.
\end{eqnarray*}
This leads, for $h$ small enough and since $\alpha'\in(0,1)$ to
\begin{eqnarray}\label{somme-delta-Z}
\sum_{n=i}^{N-1}z_n\leq
C\Big(y_N + h \sum_{n=i+1}^{N}y_n +\sum_{n=i}^{N-1} R_n\Big).
\end{eqnarray}
Iterating (\ref{majYtnZtn-2}) from $n=i$ to $n=N-1$, $i\leq N-1$, we obtain
\begin{eqnarray*}\label{}
y_i+ \frac{1+\alpha'}{2}z_i
\leq C\Big(y_N + \sum_{n=i+1}^{N-1}z_n +\sum_{n=i}^{N-1} R_n\Big).
\end{eqnarray*}
Combining (\ref{somme-delta-Z}) with the last inequality, we obtain
\begin{eqnarray*}\label{}
y_i
\leq C\Big(y_N +h\sum_{n=i+1}^{N}y_n +\sum_{n=0}^{N-1} R_n\Big).
\end{eqnarray*}
Using the discrete version of Gronwall's lemma, we get
\begin{eqnarray*}
\max_{0\leq i\leq N-1} y_i\leq C\Big(y_N +\sum_{n=0}^{N-1} R_n\Big).
\end{eqnarray*}
From Assumption \textbf{(H2)-(iii)}  we obtain
\begin{eqnarray}\label{delta-Y-Gronwall}
y_i\leq Ch(1+|x|^2) + C\sum_{n=0}^{N-1} R_n.
\end{eqnarray}
Therefore
\begin{eqnarray*}
h\sum_{n=i+1}^N y_n\leq Ch(1+|x|^2) + C\sum_{n=0}^{N-1} R_n.
\end{eqnarray*}
Inserting the last inequality into (\ref{somme-delta-Z}) and taking $i=0$  we obtain
\begin{eqnarray}\label{somme-delta-2}
\sum_{n=i}^{N-1}z_n\leq Ch(1+|x|^2) + C\sum_{n=0}^{N-1} R_n.
\end{eqnarray}
Combining (\ref{majYtnZtn}) with (\ref{delta-Y-Gronwall}) and (\ref{somme-delta-2}) the result is obtained.
\ep
%\begin{Remark}
%From the upper bound given by inequality (\ref{convergence}), we can show that, under Assumptions \textbf{(H1)} and \textbf{(H2)} and by using the estimation (\ref{apriori1yz entre s et s'}), the error (\ref{error(Y,Z)}) goes to zero when N goes to infinity. So we need Assumption \textbf{(H3)} only to obtain a rate of convergence for our scheme.
%\end{Remark}
%
\section{ Path regularity of the process $Z$}
\label{pathregualarity:section}
\setcounter{equation}{0}
\setcounter{Assumption}{0}
\setcounter{Example}{0}
\setcounter{Theorem}{0}
\setcounter{Proposition}{0}
\setcounter{Corollary}{0}
\setcounter{Lemma}{0}
\setcounter{Definition}{0}
\setcounter{Remark}{0}
The purpose of this section is to prove the $L^2$-regularity of the $Z$ component of the BDSDE's solution (\ref{BDSDE1}).
 Such a result is crucial to obtain the convergence and the rate of convergence of this numerical scheme. To this end,  the Malliavin derivatives of the solution must be introduced . This will allow us to provide a representation and regularity
   results for $Y$ and $Z$ that will immediately imply the rate of convergence of the scheme.\\
We recall the tools on the Malliavin calculus in the context of BDSDEs introduced by Pardoux and Peng \cite{par94}.
Pardoux and Peng have skipped details of this part  considering that it is just a natural extension of the work on standard BSDEs
\cite{par92}. For the sake of completeness, we give some details  which are crucial to obtaining regularity result of the process $ Z $
 and we give some technical  proofs in the Appendix.
\subsection{Malliavin calculus on the Forward SDE's}
In this section, we recall some properties on the differentiability in the Malliavin sense of the forward process $(X_{s}^{t,x})$ . Under {\bf (H3(i))}, Nualart \cite{nua}  stated that $X_{s}^{t,x} \in \mathbb{D}^{1,2}$ for any $s\in[t,T]$ and for $l\leq k$
the derivative $D_{r}^{l}X_{s}^{t,x}$ is given by:\\
(i) $D_{r}^{l}X_{s}^{t,x} = 0,\quad for\quad s<r\leq T$,\\
(ii) For any\quad$t<r\leq T$, a version of $\{D_{r}^{l}X_{s}^{t,x}, r\leq s\leq T \}$ is the unique solution of the following linear SDE
\begin{eqnarray*}
D_{r}^{l}X_{s}^{t,x} = \sigma^{l}(X_{r}^{t,x}) + \int_{r}^{s}\nabla b(X_{u}^{t,x})D_{r}^{l}X_{u}^{t,x}du
+\sum_{i=1}^{d}\int_{r}^{s}\nabla\sigma^{i}(X_{u}^{t,x})D_{r}^{l}X_{u}^{t,x}dW_{u}^{i},
\end{eqnarray*}
where $(\sigma^{i})_{i=1,\ldots,d}$ denotes the i-th column of the matrix $\sigma$.\\
Moreover, $D_{r}^{l}X_{s}^{t,x}\in \mathbb{D}^{1,2}$
for all $r, s\leq T$. For all $v\leq T$ and $l'\leq k$, we have
\begin{eqnarray*}
D_{v}^{l'}D_{r}^{l}X_{s}^{t,x}=0 \mbox{ if }s<v\vee r,
\end{eqnarray*}
and for all $s\geq v\vee r$ a version of $D_{v}^{l'}D_{r}^{l}X_{s}^{t,x}$ is the unique solution of the following SDE:
\begin{eqnarray*}
D_{v}^{l'}D_{r}^{l}X_{s}^{t,x} &= &\nabla\sigma^{l}(X_{r}^{t,x}) D_{v}^{l'}X_{r}^{t,x}
+\sum_{i=1}^{d}\nabla\sigma^{i}(X_{v}^{t,x}) D_{r}^{l}X_{v}^{t,x}\mathds{1}_{\{t\leq v \leq s\}}\\
&+&\int_{r}^{s}\Big[\sum_{j=1}^{k}\nabla((\nabla b)^{j}(X_{u}^{t,x}))D_{v}^{l'}X_{u}^{t,x}(D_{r}^{l}X_{u}^{t,x})^{j}
+\nabla b(X_{u}^{t,x})D_{v}^{l'}D_{r}^{l}X_{u}^{t,x}\Big]du \\
&+&\sum_{i=1}^{d}\int_{r}^{s}\Big[\sum_{j=1}^{k}\nabla(\nabla\sigma^{i}(X_{u}^{t,x}))^{j}D_{v}^{l'}X_{u}^{t,x}(D_{r}^{l}X_{u}^{t,x})^{j}
+\nabla\sigma^{i}(X_{u}^{t,x})D_{v}^{l'}D_{r}^{l}X_{u}^{t,x}\Big]dW_{u}^{i},
\end{eqnarray*}
where $((\nabla b)^{j})_{j=1,\ldots,k}$ (resp.$((\nabla\sigma^{i}(X_{u}^{t,x}))^{j})_{j=1,\ldots,k}$)
denotes the j-th column of the matrix $(\nabla b)$ (resp. $(\nabla\sigma^{i}(X_{u}^{t,x}))$) and $((D_{r}^{l}X_{u}^{t,x})^{j})_{j=1,\ldots,k}$
denotes the j-th component of the vector $(D_{r}^{l}X_{u}^{t,x})$.
The following inequalities will be useful later. For the proofs, we refer to Nualart \cite{nua}.
%For any $p\geq 2$, we have
%\begin{equation}\label{aprioriX}
%\|X\|_{{\cal S}^{p}_d}^{p}\leq C_{p}(1+|x|^{p}).
%\end{equation}
From Lemma 2.7 in \cite{nua} applied to $X$ and $D_{s}X$ and any $0\leq r\leq s\leq T$, there exists a constant $C$ which depends on $p$
such that we have the following inequalities
\begin{equation}\label{aprioriDX}
E\Big[\Sup_{0\leq u\leq T} ||D_s X_u||^p\Big] \leq C (1+|x|^{p}),
\end{equation}
%\begin{equation}\label{aprioriDX2}
%E[|D_{s}X_{u}-D_{s}X_{t}|^{p}]\leq C |u-t|(1+|x|^{p}),
%\end{equation}
\begin{equation}\label{aprioriDX3}
E\Big[\Sup_{s\vee r \leq u\leq T} ||D_s X_u-D_r X_u||^p\Big] \leq C |s-r|^{p/2}(1+|x|^{p}).
\end{equation}
The same argument applied for $D_{r}D_{s}X$ shows that there exists a constant $C$ which depends on $p$ such that
\begin{equation}\label{aprioriDX4}
E\Big[\Sup_{0\leq u\leq T} ||D_r D_s X_u||^p\Big] \leq C (1+|x|^{2p}).
\end{equation}
%For later use, we assume
%\begin{eqnarray*}
%&\textbf{(H3(iii))}&\quad \quad \quad \quad\Phi\in C^{1}_b(\mathbb{R}^{d},\mathbb{R}^{k}) ,f\in C
%^{1}_b([0,T]\times\mathbb{R}^{d}\times\mathbb{R}^{k}\times\mathbb{R}^{d\times k},\mathbb{R}^{k})\\
%&&\mbox{ and }  \quad\quad g\in C^{1}_b([0,T]\times\mathbb{R}^{d}\times\mathbb{R}^{k}\times\mathbb{R}^{d\times k},\mathbb{R}^{k\times l})\\
%&\textbf{(H3(iv))}&\quad\quad\quad \quad\Phi\in C^{2}_b(\mathbb{R}^{d},\mathbb{R}^{k}), f\in C^{2}_b([0,T]\times\mathbb{R}^{d}\times \mathbb{R}^{k}\times \mathbb{R}^{d\times k},\mathbb{R}^{k})\\
%&&\mbox{ and } \quad \quad g\in C^{2}_b([0,T]\times\mathbb{R}^{d}\times \mathbb{R}^{k}\times \mathbb{R}^{d\times k},\mathbb{R}^{k\times l})
%\end{eqnarray*}
\subsection{Malliavin calculus for the solution of BDSDE's}
Now, our aim is to study the differentiability in the Malliavin sense of the solution of the BDSDE \reff{1}.
We start with the following lemma which shows that a backward It\^o integral is differentiable in the
Malliavin sense if and only if its integrand is so. We recall that Pardoux and Peng \cite{par92} proved
that the result holds for the classical It\^o integral.
\begin{Lemma}\label{forD12}
Let $U\in \mathbb{H}_1^2([t,T])$ and $I_{i}(U)=\int_{t}^{T}U_{r} dW^{i}_{r},
i=1,\ldots,d$.  Then, for each $\theta\in[0,T]$ we have $U_{\theta}\in \mathbb{D}^{1,2}$
if and only if $I_{i}(U)\in\mathbb{D}^{1,2}, i=1,\ldots,d$ and for all $\theta \in [0,T]$, we have
\begin{eqnarray*}
D_{\theta}I_{i}(U)&=&\int_{\theta}^{T}D_{\theta}U_{r}dW_{r}^{i}+ U_{\theta},\,\,\theta>t,\\
D_{\theta}I_{i}(U)&=&\int_{t}^{T}D_{\theta}U_{r}dW_{r}^{i},\,\,\theta\leq t.
\end{eqnarray*}
\end{Lemma}
For backward It\^o integral, and since the Malliavin derivative is with respect to the brownian motion W, we have the following result~:
\begin{Lemma}\label{backD12}
Let $U\in \mathbb{H}_1^2([t,T])$ and $I_{i}(U)=\int_{t}^{T}U_{r} \overleftarrow{dB^{i}_{r}},
i=1,\ldots,l$. Then for each $\theta\in[0,T]$ we have $U_{\theta}\in \mathbb{D}^{1,2}$ if and only if
$I_{i}(U)\in\mathbb{D}^{1,2}, i=1,\ldots,l$ and for all $\theta \in [0,T]$, we have
\begin{eqnarray*}
D_{\theta}I_{i}(U)&=&\int_{\theta}^{T}D_{\theta}U_{r}\overleftarrow{dB^{i}_{r}},\,\,\theta>t,\\
D_{\theta}I_{i}(U)&=&\int_{t}^{T}D_{\theta}U_{r}\overleftarrow{dB^{i}_{r}},\,\,\theta\leq t.
\end{eqnarray*}
\end{Lemma}
For later use, using the same argument as in the classical BSDEs setting, we can prove the a priori estimates for the solution of the BDSDE
(see El Karoui et al. \cite{kar}).
\begin{Proposition}\label{aprioriest}
Let  $(\phi^1,f^1,g^1)$ and
$(\phi^2,f^2,g^2)$ be two standard parameters of the BDSDE \reff{1} and
$(Y^1,Z^1)$ and $(Y^2,Z^2)$ the associated solutions. Let Assumption {\bf (H2)} holds. For $s\in [t,T]$,
set $\delta Y_s:=Y^1_s-Y^2_s$,
$\delta_2f_s:=f^1(s,X_s,Y^2_s,Z^2_s)-f^2(s,X_s,Y^2_s,Z^2_s)$ and $\delta_2g_s:=g^1(s,X_s,Y^2_s,Z^2_s)-g^2(s,X_s,Y^2_s,Z^2_s)$.
Then, we have
\begin{eqnarray}\label{apriori2}
||\delta Y||_{\mathbb S^2_d([t,T])}^2+||\delta Z||_{\mathbb H^2_{k\times d}([t,T])}^2
\leq C E[|\delta Y_T|^2+\int_t^T|\delta_2f_s|^2ds+\int_t^T\|\delta_2g_s\|^2ds],
\end{eqnarray}
where $C$ is a positive constant depending only on $K$, $T$ and $\alpha$.
\end{Proposition}
%\begin{proof}
%Using the same argument as in the classical  BSDEs setting, one can prove such result (see El Karoui et al.\cite{kar}).
%\end{proof}
%\ep\\
We need also the following estimates which are deduced from the last proposition by using the Lipschitz
condition for $f$ and $g$ and Assumption \textbf{(H2-iv)}.
\begin{Lemma}\label{app2}
Let $(X^{t,x},Y^{t,x},Z^{t,x})$ be the solution of the FBDSDE \reff{forward}-\reff{1}. Then, under Assumptions {\bf(H1)}
and {\bf(H2)}, we have
\begin{eqnarray}\label{est1}
||Y^{t,x}||_{\mathbb S^2_d}+||Z^{t,x}||_{\mathbb{H}^2_{k\times d}}\leq C(1+|x|^{2}),
\end{eqnarray}
and for all $s',s\in [t,T], s' \leq s$, we have
\begin{eqnarray}\label{apriori2yz}
E\Big[\Sup_{s'\leq u\leq s}|Y_{u}^{t,x}-Y_{s'}^{t,x}|^{2}\Big]
\leq C\Big((1+|x|^{2})|s-s'|+||Z^{t,x}||_{\mathbb{H}^2_{k\times d}[s',s]}\Big).
\end{eqnarray}
%and also
%\begin{equation}\label{est2}
%E\Big[\Sup_{t\leq u\leq s}|Y_{u}-Y_{t}|^{2}\Big]
%\leq C\Big((1+|x|^{2})|s-t|+  \int_t^s  E[\|Z_r\| ^2] dr \Big).
%\end{equation}
\end{Lemma}
Now, we study the differentiability in the Malliavin sense of the solution of the BDSDE which is technical. To our knowledge,
 it does not exist in the literature.
 We have to precise that Pardoux and Peng \cite{par94} have skipped details considering that it was just an easy  extension of
the work on standard BSDEs \cite{par92}. We show in the following proposition that the derivative is a solution of a linear BDSDE
(see Peng and Pardoux \cite{par92} for the standard BSDE's and also El Karoui, Peng and Quenez (\cite {kar}, Proposition 5.3)).
 The proof is postponed to the appendix.
\begin{Proposition}\label{propderiv1}
Assume that {\bf(H1)}-{\bf(H3)} hold.
For any $t\in[0,T]$ and $x\in \mathbb{R}^{d}$, let $\{(Y_{s},Z_{s}),t\leq s\leq T\}$
denotes the unique solution of the following BDSDE:
\begin{eqnarray*}
Y_{s}= \Phi(X_{T}^{t,x}) +\int_{s}^{T}f(r,X_{r}^{t,x},Y_{r},Z_{r})dr
+ \int_{s}^{T}g(r,X_{r}^{t,x},Y_{r},Z_{r})
\overleftarrow{dB_{r}} - \int_{s}^{T}Z_{r}dW_{r},\,\,t\leq s\leq T.
\end{eqnarray*}
Then,
 $(Y, Z) \in \Bc^2([t,T],\mathbb{D}^{1,2})$ and $\{D_{\theta}Y_{s},D_{\theta}Z_{s};t\leq s,\theta\leq T\}$ is given by:\\
(i) $D_{\theta}Y_{s}=0, D_{\theta}Z_{s}=0$ for all  $t\leq s<\theta\leq T$\\
(ii) for any fixed $\theta\in [t,T]$, $\theta\leq s\leq T$ and $1\leq i\leq d$, a version of $(D_{\theta}^{i}Y_{s},D_{\theta}^{i}Z_{s})$ is the unique
solution of the following BDSDE:
\begin{eqnarray}\label{DY}
D_{\theta}^{i}Y_{s}&=&\nabla\Phi(X_{T}^{t,x})D_{\theta}^{i}X_{T}^{t,x}+\int_{s}^{T}\Big( \nabla_{x}f(r,X_{r}^{t,x},Y_{r},Z_{r})D_{\theta}^{i}X_{r}^{t,x}\Big)dr\nonumber\\
&+&\int_{s}^{T}\Big( \nabla_{y}f(r,X_{r}^{t,x},Y_{r},Z_{r})D_{\theta}^{i}Y_{r}
+\Sum_{j=1}^{d}\nabla_{z^j}f(r,X_{r}^{t,x},Y_{r},Z_{r})
 D_{\theta}^{i}Z_{r}^j\Big)dr\nonumber\\
&+&\Sum_{n=1}^{l}\int_{s}^{T}\Big( \nabla_{x}g^n(r,X_{r}^{t,x},Y_{r},Z_{r})D_{\theta}^{i}X_{r}^{t,x}
+\nabla_{y}g^n(r,X_{r}^{t,x},Y_{r},Z_{r})D_{\theta}^{i}Y_{r}\Big)\overleftarrow{dB_{r}^{n}}\nonumber\\
&+&\Sum_{n=1}^{l}\int_{s}^{T}\Sum_{j=1}^{d}\Big( \nabla_{z^j}g^n(r,X_{r}^{t,x},Y_{r},Z_{r})
D_{\theta}^{i}Z_{r}^j\Big)\overleftarrow{dB_{r}^n}- \int_{s}^{T}\Sum_{j=1}^{d}D_{\theta}^{i}Z_{r}^jdW_{r}^j,
\end{eqnarray}
where $(z^j)_{1\leq j\leq d}$ denotes the j-th column of the matrix $z$,
$(g^n)_{1\leq n\leq l}$ denotes the n-th column of the matrix g and $B=(B^1,\ldots,B^l)$.
\end{Proposition}
%\begin{proof}
%See Appendix.
%\end{proof}
%\ep\\

The second order differentiability in the Malliavin sense of the solution of the BDSDE will be given in Appendix.

\subsection{Representation results for BDSDEs}
%\label{representation:subsection}
%\setcounter{equation}{0}
%\setcounter{Assumption}{0}
%\setcounter{Example}{0}
%\setcounter{Theorem}{0}
%\setcounter{Proposition}{0}
%\setcounter{Corollary}{0}
%\setcounter{Lemma}{0}
%\setcounter{Definition}{0}
%\setcounter{Remark}{0}
%\subsection{Representation}
In this subsection, we will prove a representation result of $(Z,DZ)$ which will be useful to prove the rate
of convergence of our numerical scheme.
\begin{Proposition}\label{representationZ}
Assume that {\bf (H1)}-{\bf (H3)} hold. Then, for $t\leq s\leq T$, we have
%\begin{eqnarray}\label{zs2}
%D_{s}Y_{s}=Z_{s},\quad P-a.s.,
%\textrm{  and  }\quad
%\|Z\|_{{ \mathbb S}^{2}_{k\times d}([t,T])}^{2}\leq C(1+|x|^{2}).
%\end{eqnarray}
\vspace{-0.25cm}
%\begin{center}
\begin{multicols}{2}
   \begin{equation}\label{zs1}
     D_{s}Y_{s}^{t,x}=Z_{s}^{t,x},\quad P-a.s.,
\end{equation}
\hspace{-1cm}
\begin{equation}\label{zs2}
\textrm{and}\qquad\|Z^{t,x}\|_{{ \mathbb S}^{2}_{k\times d}([t,T])}^{2}\leq C(1+|x|^{2}).
\end{equation}
\end{multicols}
%\end{center}
%For $l_{1}, l_{2}\leq d$, $t\leq s\leq T$, we have
%\begin{eqnarray}\label{eqdz}
% D_s^{l_2}D_t^{l_1}Y_s=D_t^{l_2}Z_s^{l_1}, \quad \mathbb{P}.a.s.,
%\end{eqnarray}
%and
%\begin{eqnarray}\label{ddy}
%\|D_{s}^{l_1}Z\|_{{\mathbb S}^{2}_{k\times d}([t,T])}^{2}\leq C(1+|x|^{4}).
%\end{eqnarray}
\end{Proposition}
\begin{proof}
To simplify the notations, we restrict ourselves to the case $k=d=1$.\\
%1.
 Notice that for $t\leq s$, we have
\begin{eqnarray*}
Y_{s}^{t,x}=Y_{t}^{t,x}-\int_{t}^{s}f(r,\Sigma_{r}^{t,x})dr - \int_{t}^{s}g(r,\Sigma_{r}^{t,x})
\overleftarrow{dB_{r}}+\int_{t}^{s}Z_{r}^{t,x}dW_{r},
\end{eqnarray*}
where $\Sigma_{r}^{t,x}:=(X_{r}^{t,x},Y_{r}^{t,x},Z_{r}^{t,x})$.\\
It follows from Lemma \ref{forD12} and Lemma \ref{backD12} that, for $t<\theta \leq s$
\begin{eqnarray*}
D_{\theta}Y_{s}^{t,x}&=&Z_{\theta}^{t,x}-\int_{\theta}^{s}\Big( \nabla_{x}f(r,\Sigma_{r}^{t,x})D_{\theta}X_{r}^{t,x}
+ \nabla_{y}f(r,\Sigma_{r}^{t,x})D_{\theta}Y_{r}^{t,x}+\nabla_{z}f(r,\Sigma_{r}^{t,x})
D_{\theta}Z_{r}^{t,x}\Big)dr\\
&-&\int_{\theta}^{s}\Big( \nabla_{x}g(r,\Sigma_{r}^{t,x})
D_{\theta}X_{r}^{t,x}+\nabla_{y}g(r,\Sigma_{r}^{t,x})D_{\theta}Y_{r}^{t,x}+\nabla_{z}g(r,\Sigma_{r}^{t,x})
D_{\theta}Z_{r}^{t,x}\Big)\overleftarrow{dB_{r}}
%-\int_{\theta}^{s}\nabla_{z}g(r,\Sigma_{r})
%D_{\theta}Z_{r}\overleftarrow{dB_{r}}
+ \int_{\theta}^{s}D_{\theta}Z_{r}^{t,x}dW_{r}.
\end{eqnarray*}
Then by taking $\theta=s$, it follows that equality (\ref{zs1}) holds.\\
% From \reff{aprioriDY}, we deduce that
%\reff{zs2} holds.\\
From Proposition \ref{propderiv1} and inequalities \reff{apriori1} and \reff{aprioriDX}, we deduce that for each $\theta\leq T$
\begin{equation}\label{aprioriDY}
E[\Sup_{t\leq s\leq T}|D_{\theta}Y_s^{t,x}|^2]+E[\int_t^T|D_{\theta}Z_s^{t,x}|^2ds]\leq C (1+|x|^{2}).
\end{equation}
Then, by taking $\theta=s$, we deduce that \reff{zs2} holds.
%2. Notice that for $\theta\leq t\leq s$
% \begin{eqnarray*}
%D_{\theta}Y_{s}&=&D_{\theta}Y_{t}
%-\int_{t}^{s}\Big( \nabla_{x}f(r,\Sigma_{r})D_{\theta}X_{r}+\nabla_{y}f(r,\Sigma_{r})D_{\theta}Y_{r}+\nabla_{z}f(r,\Sigma_{r})
% D_{\theta}Z_{r}\Big)dr\nonumber\\
%&-&\int_{t}^{s}\Big( \nabla_{x}g(r,\Sigma_{r})D_{\theta}X_{r}
%+\nabla_{y}g(r,\Sigma_{r})D_{\theta}Y_{r}+\nabla_{z}g(r,\Sigma_{r})
%D_{\theta}Z_{r}\Big)\overleftarrow{dB_{r}}+ \int_{t}^{s}D_{\theta}Z_{r}dW_{r}.\nonumber
%\end{eqnarray*}
%It follows from Lemma \ref{forD12} and Lemma \ref{backD12} that, for $\theta\leq t<v\leq s$
%\begin{eqnarray*}
%D_vD_{\theta}Y_{s}&=&D_{\theta}Z_{v}
%-\int_{v}^{s}D_{v}(\Sigma_{r})^*[Hf](r,\Sigma_{r})D_{\theta}(\Sigma_{r})dr
%-\int_{v}^{s}\nabla f(r,\Sigma_{r})D_{v}D_{\theta}(\Sigma_{r})dr\nonumber\\
%&-&\int_{v}^{s}D_{v}(\Sigma_{r})^*[Hg](r,\Sigma_{r})D_{\theta}(\Sigma_{r})\overleftarrow{dB_{r}}
%-\int_{v}^{s}\nabla g(r,\Sigma_{r})D_{v}D_{\theta}(\Sigma_{r})\overleftarrow{dB_{r}}\\
%&+& \int_{v}^{s}D_vD_{\theta}Z_{r}dW_{r},
%\end{eqnarray*}
%where $[Hf]$ and $[Hg]$ are respectively the Hessian of $f$ and $g$.\\
%Then by taking $v=s$ and $t=\theta$, it follows that equality \reff{eqdz} holds.
%We have from estimate \reff{apriori1} and inequality \reff{aprioriDX4},
%that for each $v\leq T$ and $\theta\leq T$
%\begin{equation}\label{aprioriDDY}
%E[\Sup_{t\leq s\leq T}|D_vD_{\theta}Y_s|^2]+E[\int_t^T|D_vD_{\theta}Z_s|^2ds]\leq C (1+|x|^{4}).
%\end{equation}
%and then by taking $v=s$ and $t=\theta$ we deduce that \reff{ddy} holds.
\end{proof}
\ep\\
%\begin{Proposition}
%Assume that {\bf (H2)} and {\bf (H4(ii))} hold.
%For $l_1,l_2\leq d$, $0\leq s\leq T$, we have
%\begin{eqnarray}\label{eqdz}
%D_{s}^{l_2}D_{t}^{l_1}Y_{s}&=&D_{t}^{l_1}Z_{s},
%\end{eqnarray}
%and we have
%\begin{eqnarray}\label{ddy}
%\|D_{s}^{l_1}Z\|_{S^{2}_{k\times d}(\mathbb{D}^{1,2})}^{2}\leq C(1+|x|^{4})
%\end{eqnarray}
%\end{Proposition}
%%%%%%%%%%%%%%%%%%%%%%%%%%%%%%%%%%%%%%%%%%% Path regularity %%%%%%%%%%%%%%%%%%%%%%%%%%%%%%%%%%%%%%%%%%%%%%
\subsection{Path regularity}
In this subsection, we extend the result of Zhang \cite{zha} which concerns the $L^2$-regularity
of the martingale integrand $Z$. Such result is crucial to derive the rate of convergence of our numerical scheme.
We start with the following proposition which gives an upper bound for $$E\Big[\Sup_{r\in[s,u]}|Y_{r}^{t,x}-Y_{s}^{t,x}|^{2} \Big]
 \quad \mbox{and} \quad  E\Big[||Z_{u}^{t,x}-Z_{s}^{t,x}||^{2} \Big], \quad t\leq s\leq u\leq T.$$
%%%%%%%%%%%%%%%%%%%%%%%%%%%%%%%%%%%%%%%%%%%%%%%%%%% Proposition %%%%%%%%%%%%%%%%%%%%%%%%%%%%%%%%%%%%%%%%%%%%%%%%%%%%
\begin{Proposition}\label{cvzts}
Assume that {\bf(H1)}-{\bf(H3)} hold. Then for $t\leq s\leq u\leq T$, we have
\begin{eqnarray}
E\Big[\sup_{r\in[s,u]}|Y_{r}^{t,x}-Y_{s}^{t,x}|^{2} \Big]&\leq& C (1+|x|^{2})|u-s|,\label{yrs}\\
E\Big[||Z_{u}^{t,x}-Z_{s}^{t,x}||^{2} \Big]&\leq& C (1+|x|^{2})|u-s|.\label{zrs}
\end{eqnarray}
\end{Proposition}
%%%%%%%%%%%%%%%%%%%%%%%%%%%%%%%%%%%%%%%%%%%%%%%%%%%%%% Proof  %%%%%%%%%%%%%%%%%%%%%%%%%%%%%%%%%%%%%%%%%%%%%%%
\begin{proof}
To simplify the notations, we restrict ourselves to the case $k=d=l=1$.\\
(i) Plugging inequality \reff{zs2} in the estimate \reff{apriori2yz}, the result \reff{yrs} holds.\\
(ii) From Proposition \ref{representationZ}, we have
\begin{eqnarray}\label{diffz}
E\Big[|Z_{u}^{t,x}-Z_{s}^{t,x}|^{2} \Big]&\leq& C E[|D_{u}Y_{u}^{t,x}-D_{s}Y_{u}^{t,x}|^{2}]+C E[|D_{s}Y_{u}^{t,x}-D_{s}Y_{s}^{t,x}|^{2}].
\end{eqnarray}
From the definition of the BDSDE (\ref{DY}), we have
\begin{eqnarray*}
&&D_{u}Y_{u}^{t,x}-D_{s}Y_{u}^{t,x}=\nabla\Phi(X_{T}^{t,x})(D_{u}X_{T}^{t,x}-D_{s}X_{T}^{t,x})
+\int_{u}^{T}\Big( \nabla_{x}f(r,\Sigma_{r}^{t,x})(D_{u}X_{r}^{t,x}-D_{s}X_{r}^{t,x})\Big)dr\nonumber\\
&+&\int_{u}^{T}\Big( \nabla_{y}f(r,\Sigma_{r}^{t,x})(D_{u}Y_{r}^{t,x}-D_{s}Y_{r}^{t,x})
+\nabla_{z}f(r,\Sigma_{r}^{t,x})(D_{u}Z_{r}^{t,x}-D_{s}Z_{r}^{t,x})\Big)dr\nonumber\\
&+&\int_{u}^{T}\Big( \nabla_{x}g(r,\Sigma_{r}^{t,x})(D_{u}X_{r}^{t,x}-D_{s}X_{r}^{t,x})
+\nabla_{y}g(r,\Sigma_{r}^{t,x})(D_{u}Y_{r}^{t,x}-D_{s}Y_{r}^{t,x})\Big)\overleftarrow{dB_{r}}\nonumber\\
&+&\int_{u}^{T}\Big( \nabla_{z}g(r,\Sigma_{r}^{t,x})
(D_{u}Z_{r}^{t,x}-D_{s}Z_{r}^{t,x})\Big)\overleftarrow{dB_{r}}
-\int_{u}^{T}(D_{u}Z_{r}^{t,x}-D_{s}Z_{r}^{t,x})dW_{r}.
\end{eqnarray*}
Applying the generalized It\^o's formula (see \cite{par94}, Lemma 1.3), we obtain
\begin{eqnarray*}
&&|D_{u}Y_{T}^{t,x}-D_{s}Y_{T}^{t,x}|^2-|D_{u}Y_{u}^{t,x}-D_{s}Y_{u}^{t,x}|^2=\\
&-&2\int_{u}^{T} \nabla_{x}f(r,\Sigma_{r}^{t,x})(D_{u}X_{r}^{t,x}-D_{s}X_{r}^{t,x})(D_{u}Y_{r}^{t,x}-D_{s}Y_{r}^{t,x})dr
-2\int_{u}^{T} \nabla_{y}f(r,\Sigma_{r}^{t,x})(D_{u}Y_{r}^{t,x}-D_{s}Y_{r}^{t,x})^2dr\nonumber\\
&-&2\int_{u}^{T} \nabla_{z}f(r,\Sigma_{r}^{t,x})(D_{u}Z_{r}^{t,x}-D_{s}Z_{r}^{t,x})(D_{u}Y_{r}^{t,x}-D_{s}Y_{r}^{t,x})dr\nonumber\\
&-&2\int_{u}^{T} \nabla_{x}g(r,\Sigma_{r}^{t,x})(D_{u}X_{r}^{t,x}-D_{s}X_{r}^{t,x})(D_{u}Y_{r}^{t,x}-D_{s}Y_{r}^{t,x})
\overleftarrow{dB_{r}}\nonumber\\
&-&2\int_{u}^{T}\nabla_{y}g(r,\Sigma_{r}^{t,x})(D_{u}Y_{r}^{t,x}-D_{s}Y_{r}^{t,x})^2\overleftarrow{dB_{r}}\nonumber\\
&-&2\int_{u}^{T} \nabla_{z}g(r,\Sigma_{r}^{t,x})
(D_{u}Z_{r}^{t,x}-D_{s}Z_{r}^{t,x})(D_{u}Y_{r}^{t,x}-D_{s}Y_{r}^{t,x})\overleftarrow{dB_{r}}\\
&+& 2\int_{u}^{T}(D_{u}Z_{r}^{t,x}-D_{s}Z_{r}^{t,x})(D_{u}Y_{r}^{t,x}-D_{s}Y_{r}^{t,x})dW_{r}\\
&-&\int_{u}^{T} \big|\nabla_{x}g(r,\Sigma_{r}^{t,x})(D_{u}X_{r}^{t,x}-D_{s}X_{r}^{t,x})
+ \nabla_{y}g(r,\Sigma_{r}^{t,x})(D_{u}Y_{r}^{t,x}-D_{s}Y_{r}^{t,x})
+\nabla_{z}g(r,\Sigma_{r}^{t,x})(D_{u}Z_{r}^{t,x}-D_{s}Z_{r}^{t,x}\big|^2dr\\
&+&\int_{u}^{T}|D_{u}Z_{r}^{t,x}-D_{s}Z_{r}^{t,x}|^2dr.
\end{eqnarray*}
From inequalities \reff{aprioriDY} and \reff{aprioriDX}, using the Burkholder-Davis-Gundy's inequality
and Assumption {\bf(H2)}, the stochastic integrals which appear in the last equation disappear when we take the expectation.\\
By Young inequality, we obtain, for $\epsilon '>0$
\begin{eqnarray*}
&&E[|D_{u}Y_{u}^{t,x}-D_{s}Y_{u}^{t,x}|^2]+E[\int_u^T|D_{u}Z_{r}^{t,x}-D_{s}Z_{r}^{t,x}|^2]dr \leq
E[|\nabla\Phi(X_{T}^{t,x})(D_{u}X_{T}^{t,x}-D_{s}X_{T}^{t,x})|^2]\\
&+& 2E[\int_u^T\nabla_{x}f(r,\Sigma_{r}^{t,x})(D_{u}X_{r}^{t,x}-D_{s}X_{r}^{t,x})(D_{u}Y_{r}^{t,x}-D_{s}Y_{r}^{t,x})dr]\\
&+&2E[\int_{u}^{T} \nabla_{y}f(r,\Sigma_{r}^{t,x})(D_{u}Y_{r}^{t,x}-D_{s}Y_{r}^{t,x})^2 dr]\\
&+&2E[\int_{u}^{T}\nabla_{z}f(r,\Sigma_{r}^{t,x})(D_{u}Z_{r}^{t,x}-D_{s}Z_{r}^{t,x})(D_{u}Y_{r}^{t,x}-D_{s}Y_{r}^{t,x})dr]\\
&+&C(1+\frac{1}{\epsilon '})E[\int_{u}^{T} \nabla_{x}g(r,\Sigma_{r}^{t,x})^2|D_{u}X_{r}^{t,x}-D_{s}X_{r}^{t,x}|^2dr]\\
&+&C(1+\frac{1}{\epsilon '})E[\int_{u}^{T} \nabla_{y}g(r,\Sigma_{r}^{t,x})^2|D_{u}Y_{r}^{t,x}-D_{s}Y_{r}^{t,x}|^2dr]\\
&+&(1+\epsilon ')E[\int_{u}^{T} \nabla_{z}g(r,\Sigma_{r}^{t,x})^2|D_{u}Z_{r}^{t,x}-D_{s}Z_{r}^{t,x}|^2dr].
\end{eqnarray*}
Hence by using Assumption {\bf(H2)} and Young inequality, we have for $\epsilon, \epsilon'>0$ and $C>0$,
\begin{eqnarray*}
&&E[|D_{u}Y_{u}^{t,x}-D_{s}Y_{u}^{t,x}|^2]+E[\int_u^T|D_{u}Z_{r}^{t,x}-D_{s}Z_{r}^{t,x}|^2dr]\leq
K^2 E[|D_{u}X_{T}^{t,x}-D_{s}X_{T}^{t,x}|^2] \\
&+& 2K E[\int_u^T |D_{u}X_{r}^{t,x}-D_{s}X_{r}^{t,x}|^2 dr]+4K E[\int_u^T |D_{u}Y_{r}^{t,x}-D_{s}Y_{r}^{t,x}|^2 dr]\\
&+& K\epsilon E[\int_u^T |D_{u}Y_{r}^{t,x}-D_{s}Y_{r}^{t,x}|^2 dr]+\frac{K}{\epsilon}E[\int_u^T |D_{u}Z_{r}^{t,x}-D_{s}Z_{r}^{t,x}|^2 dr]\\
&+& C K^2 (1+\frac{1}{\epsilon '})E[\int_u^T |D_{u}X_{r}^{t,x}-D_{s}X_{r}^{t,x}|^2 dr]+ CK^2(1+\frac{1}{\epsilon '})E[\int_u^T |D_{u}Y_{r}^{t,x}-D_{s}Y_{r}^{t,x}|^2 dr]\\
&+&(1+\epsilon ')\alpha^2 E[\int_u^T |D_{u}Z_{r}^{t,x}-D_{s}Z_{r}^{t,x}|^2 dr].
\end{eqnarray*}
Then, we obtain
\begin{eqnarray*}
&&E[|D_{u}Y_{u}^{t,x}-D_{s}Y_{u}^{t,x}|^2]+E[\int_u^T|D_{u}Z_{r}^{t,x}-D_{s}Z_{r}^{t,x}|^2dr]\leq
K^2 E[|D_{u}X_{T}^{t,x}-D_{s}X_{T}^{t,x}|^2] \\
&+& K(2+K C(1+\frac{1}{\epsilon '}))E[\int_u^T |D_{u}X_{r}^{t,x}-D_{s}X_{r}^{t,x}|^2 dr]\\
&+& (K^2 C(1+\frac{1}{\epsilon '})+(4+\epsilon)K) E[\int_u^T |D_{u}Y_{r}^{t,x}-D_{s}Y_{r}^{t,x}|^2 dr]\\
&+ &((1+\epsilon ')\alpha^2+\frac{K}{\epsilon}) E[\int_u^T |D_{u}Z_{r}^{t,x}-D_{s}Z_{r}^{t,x}|^2 dr].
\end{eqnarray*}
For $\epsilon$ large enough and $\epsilon '$ small enough, we have $(1+\epsilon ')\alpha^2+\frac{K}{\epsilon}<1$.
From inequality \reff{aprioriDX3}, we deduce that
\begin{eqnarray*}
E[|D_{u}Y_{u}^{t,x}-D_{s}Y_{u}^{t,x}|^2]\leq  C\Big( (1+|x|^2)|u-s| + E[\int_u^T |D_{u}Y_{r}^{t,x}-D_{s}Y_{r}^{t,x}|^2 dr]\Big),
\end{eqnarray*}
where $C$ is a positive constant.\\
From Gronwall's lemma we have
\begin{eqnarray}\label{diffdy1}
E[|D_{u}Y_{u}^{t,x}-D_{s}Y_{u}^{t,x}|^2]\leq  C (1+|x|^2)|u-s|.
\end{eqnarray}
Since $(D_{s}Y_{u}^{t,x})_{s\leq u\leq T}$ satisfies the BDSDE (\ref{DY}),
inequalities \reff{apriori2yz}-\reff{zs2} hold for
$(D_{s}Y_{u}^{t,x},D_{s}Z_{u}^{t,x})_{s\leq u\leq T}$ and yield
\begin{eqnarray}\label{diffdy2}
E[|D_{s}Y_{u}^{t,x}-D_{s}Y_{s}^{t,x}|^{2}]\leq C (1+|x|^{2})|u-s|.
\end{eqnarray}
Plugging \reff{diffdy1} and \reff{diffdy2} into \reff{diffz}, we obtain \reff{zrs}.
\end{proof}
\ep\\
\subsection{Application to the scheme's convergence}
The following theorem states the rate of convergence of our numerical scheme.
\begin{thm}\label{ratecvtheo}
Under Assumptions {\bf(H1)}-{\bf(H3)}, there exists a positive constant C
(depending only on $T$, $K$, $\alpha$, $|b(0)|$, $||\sigma(0)||$, $|f(t,0,0,0)|$ and $||g(t,0,0,0)||$) such that
\begin{eqnarray}\label{ratecv}
Error_{N}(Y,Z)\leq C h (1+|x|^{2}).
%\frac{C}{N}(1+|x|^{2}).
\end{eqnarray}
\end{thm}
\begin{proof}
We recall that from Theorem \ref{upper-bound}, we have
\begin{eqnarray*}
Error_{N}(Y,Z) &\leq&  C h(1+|x|^2) + C \sum_{n=0}^{N-1} \int_{t_{n}}^{t_{n+1}}E[||Z_{s}-\bar{Z}_{t_{n}}||^{2}]ds \nonumber\\
&+&  C \sum_{n=0}^{N-1}\! \int_{t_{n}}^{t_{n+1}}\!\!\!\!\!\!\!\!\!\! E[||Z_{s}-\bar{Z}_{t_{n+1}}||^{2}]ds
+ C \sum_{n=0}^{N-1} \int_{t_{n}}^{t_{n+1}}E[|Y_{s}-Y_{t_{n}}|^{2}]ds\nonumber\\
&+& C \sum_{n=0}^{N-1} \int_{t_{n}}^{t_{n+1}}E[|Y_{s}-Y_{t_{n+1}}|^{2}]ds.
\end{eqnarray*}
\textbf{First step:} We deal with the $Y$ part. We have
\begin{eqnarray*}
\sum_{n=0}^{N-1} \int_{t_{n}}^{t_{n+1}}E[|Y_{s}-Y_{t_{n}}|^{2}]ds
\leq \sum_{n=0}^{N-1} \int_{t_{n}}^{t_{n+1}}E[\sup_{t_n\leq s\leq t_{n+1}}|Y_{s}-Y_{t_{n}}|^{2}]ds.
\end{eqnarray*}
From inequality \eqref{yrs} (see Proposition \ref{cvzts}), we obtain
\begin{eqnarray}\label{majoration-Y_s-Y_t_n}
\sum_{n=0}^{N-1} \int_{t_{n}}^{t_{n+1}}E[|Y_{s}-Y_{t_{n}}|^{2}]ds
\leq C h (1+|x|^2).
\end{eqnarray}
Similarly, we get
\begin{eqnarray}\label{majoration-Y_s-Y_t_n+1}
\sum_{n=0}^{N-1} \int_{t_{n}}^{t_{n+1}}E[|Y_{s}-Y_{t_{n+1}}|^{2}]ds
\leq C h (1+|x|^2).
\end{eqnarray}
\textbf{Second step:}
From the definition (\ref{barZdef}), ${\bar Z_{t_{n}}}$ is the best approximation of $(Z_{t})_{t_{n}\leq t < t_{n+1}}$ by $\mathcal{F}_{t_{n}}$-measurable random variable in the following sense
\begin{eqnarray}
E\bigg[\int_{t_{n}}^{t_{n+1}}\|Z_{s}-\bar{Z}_{t_{n}} \|^{2}ds \bigg]=
\inf_{Z_{n}\in L^{2}(\Omega ,\mathcal{F}_{t_n})}E\bigg[\int_{t_{n}}^{t_{n+1}}\|Z_{s}-Z_{n} \|^{2}ds \bigg]\nonumber
\end{eqnarray}
%which implies
%\begin{eqnarray}\label{barzproperty}
%E[\|Z_{s}-\bar{Z}_{t_{n}}\|^{2}]\leq E[\|Z_{s}-Z_{t_{n}}\|^{2}].\nonumber
%\end{eqnarray}
From the estimation \eqref{zrs} (see Proposition \ref{cvzts}), we have
\begin{eqnarray}\label{majoration-Z_s-Z_t_n}
E\Big[||Z_{s}-Z_{t_n}||^{2} \Big]\leq C (1+|x|^{2})|s-t_n|\leq C h (1+|x|^{2}),
\end{eqnarray}
for all $s\in[t_n,t_{n+1}]$ and $0\leq n\leq N-1$ where C depends only on $T$, $K$, $b(0)$, $\sigma(0)$, $f(t,0,0,0)$
and $g(t,0,0,0)$. Then
\begin{eqnarray}\label{sommeZbar}
\Sum_{n=0}^{N-1}E\Big[\int_{t_n}^{t_{n+1}}||Z_{s}-\bar{Z}_{t_{n}}||^{2}ds \Big]\leq C h (1+|x|^{2}).\nonumber
\end{eqnarray}
%Integrating over $s$ and summing over $i$, it yields
%\begin{eqnarray*}
%\Sum_{i=0}^{N-1}E\Big[\int_{t_i}^{t_{i+1}}||Z^m_{s}-Z^m_{t_i}||^{2}ds \Big]\leq C h (1+|x|^{2}).
%\end{eqnarray*}
%%where $C(x)=C(1+|x|^{2})$.\\
%Sending m to infinity, we deduce that
%\begin{eqnarray}\label{}
%\Sum_{i=0}^{N-1}E\Big[\int_{t_i}^{t_{i+1}}||Z_{s}-Z_{t_i}||^{2}ds \Big]\leq C h (1+|x|^{2}).
%\end{eqnarray}
%Now applying  the last inequality for $Z^N$, we obtain that
%\begin{eqnarray}\label{majZsN}
%\Sum_{i=0}^{N-1}E\Big[\int_{t_i}^{t_{i+1}}||Z^N_{s}-Z^N_{t_i}||^{2}ds \Big]\leq C h (1+|x|^{2}).
%\end{eqnarray}
On the other hand, we have
\begin{eqnarray}\label{majZbartn+1}
 E\Big[\int_{t_n}^{t_{n+1}}\!\!\!||Z_{s}-\bar{Z}_{t_{n+1}}||^{2}ds \Big]\leq
 2 E\Big[\int_{t_n}^{t_{n+1}}\!\!\!||Z_{s}-Z_{t_{n+1}}||^{2}ds \Big]+2 E\Big[\int_{t_n}^{t_{n+1}}\!\!\!||Z_{t_{n+1}}-\bar{Z}_{t_{n+1}}||^{2}ds \Big].
 \end{eqnarray}
 From the definition of $\bar{Z}_{t_{n+1}}$ and the Jensen's inequality, we have
 \begin{eqnarray*}
 E\Big[||Z_{t_{n+1}}-\bar{Z}_{t_{n+1}}||^{2} \Big]
 &=&E\Big[||Z_{t_{n+1}}-\frac{1}{h}E_{t_{n+1}}\Big[\int_{t_{n+1}}^{t_{n+2}}Z_sds\Big]||^{2} \Big]\\
 &=&E\Big[||\frac{1}{h}E_{t_{n+1}}\Big[\int_{t_{n+1}}^{t_{n+2}}(Z_{t_{n+1}}-Z_s)ds\Big]||^{2} \Big]\\
 &\leq&\frac{1}{h^2}E\Big[||\int_{t_{n+1}}^{t_{n+2}}(Z_{t_{n+1}}-Z_s)ds||^{2}\Big].
 \end{eqnarray*}
 By using Cauchy Schwartz inequality, we obtain
 \begin{eqnarray*}
 E\Big[||Z_{t_{n+1}}-\bar{Z}_{t_{n+1}}||^{2} \Big]&\leq& \frac{1}{h^2} E\Big[h \int_{t_{n+1}}^{t_{n+2}}||Z_{t_{n+1}}-Z_s||^{2}ds\Big]\\
 &\leq& \frac{1}{h} \int_{t_{n+1}}^{t_{n+2}}E\Big[||Z_{t_{n+1}}-Z_s||^{2}\Big]ds
  \end{eqnarray*}
Using the estimation \eqref{zrs}, we get
\begin{eqnarray*}
E\Big[||Z_{t_{n+1}}-\bar{Z}_{t_{n+1}}||^{2} \Big]&\leq&\frac{1}{h} \int_{t_{n+1}}^{t_{n+2}}C(1+|x|^2)|s-t_{n+1}|ds\\
&\leq& Ch(1+|x|^2).
\end{eqnarray*}
Inserting the last inequality in (\ref{majZbartn+1}) and using again the estimate \eqref{zrs}, we obtain
%From (\ref{majZbartn+1}), we obtain
\begin{eqnarray}\label{sommeZN}
\Sum_{n=0}^{N-2}E\Big[\int_{t_n}^{t_{n+1}}||Z_{s}-\bar{Z}_{t_{n+1}}||^{2}ds \Big]\leq C h (1+|x|^{2}).\nonumber
\end{eqnarray}
Using the estimation \eqref{zs2}, we obtain
\begin{eqnarray*}
E\Big[\int_{t_{N-1}}^{t_{N}}||Z_{s}||^{2}ds \Big]\leq C h (1+|x|^{2}).\nonumber
\end{eqnarray*}
Then
\begin{eqnarray}\label{majoration-Z_s-Z_t_n+1}
\Sum_{n=0}^{N-1}E\Big[\int_{t_n}^{t_{n+1}}||Z_{s}-\bar{Z}_{t_{n+1}}||^{2}ds \Big]\leq C h (1+|x|^{2}).
\end{eqnarray}
%Plugging (\ref{sommeZN}) and (\ref{majZsN}) in (\ref{ErrorYErrorZ}), we deduce inequality \reff{ratecv}.
Finally, plugging (\ref{majoration-Y_s-Y_t_n}), (\ref{majoration-Y_s-Y_t_n+1}), (\ref{majoration-Z_s-Z_t_n}) and (\ref{majoration-Z_s-Z_t_n+1})
in (\ref{convergence}) in  Theorem \ref{upper-bound}, we get
\begin{eqnarray*}
Error_{N}(Y,Z)\leq C h (1+|x|^{2}).
\end{eqnarray*}

\end{proof}
\ep\\

\section{Numerical scheme for the weak solution of the SPDE}
\setcounter{equation}{0}
\setcounter{Assumption}{0}
\setcounter{Example}{0}
\setcounter{Theorem}{0}
\setcounter{Proposition}{0}
\setcounter{Corollary}{0}
\setcounter{Lemma}{0}
\setcounter{Definition}{0}
\setcounter{Remark}{0}

Most numerical works on SPDEs are concentrated  on the Euler finite-difference scheme (see   \cite{gyo1}, \cite{gyo2} , \cite{gyo}),
 on finite element method (see  \cite{Walsh}) and also on  spectral Galerkin methods (see \cite{Kloeden} and the references therein).
   Here,  we follow
a  probabilistic method  based on  the Feynman-Kac's formula for the weak solution of  the semilinear SPDE (\ref{SPDE})
based on BSDE approach (see \cite{BMat}, \cite{MS02}).  We consider a weak Sobolev solution of such SPDE in the sense
 that $u$ shall be considered  as a predictable process in some first order Sobolev space.  Therefore, we improve the
  convergence and the rate of convergence  of the $L^2$-norm  error  of such solution by using the convergence results on BDSDEs
   proved in section 4.
   % and an equivalence norm result.
   % given in Barles and Lesigne \cite{BL} and Bally and
%Matoussi \cite{BMat}.

\subsection{Weak solution for SPDE}
\label{weakSPDE:subsection}
Since we work on the whole space $\mathbb R^d$, we introduce a weight function $\rho$ satisfying the following conditions :  $ \rho $
  is  a positive  locally integrable function ,   $\frac{1}{\rho}$ is locally integrable and $\int_{\mathbb{R}^{d}}(1+|x|^{2})\rho(x)dx <\infty$. For example, we can take $\rho(x)\!=\!e^{-\frac{x^{2}}{2}}$ or $\rho(x)\!=\!e^{-|x|}$.
As a consequence of $\bf{(H3)}$, we have $\displaystyle{\int_{\mathbb{R}^{d}}|\Phi(x)|^2\rho(x)dx<\infty}$, $\int_{0}^{T}\int_{\mathbb{R}^{d}}|f(t,x,0,0)|^2\rho(x)dxdt<\infty$ and $\int_{0}^{T}\int_{\mathbb{R}^{d}}|g(t,x,0,0)|^2\rho(x)dxdt<\infty$.\\

We denote by  $L^{2}(\mathbb{R}^{d},\rho(x)dx)$
the weighted Hilbert space and we employ the following notation for its scalar product and its norm: $(u,v)_{\rho}=\int_{\mathbb{R}^{d}}u(x)v(x)\rho(x)dx$ and $\|u\|_{\rho}=(u,u)_{\rho}^{\frac{1}{2}}$. Then, we define by
$\mathrm{H}_{\rho} ^{1}(\mathbb{R}^{d})$ the associated weighted first order Dirichlet space and its norm $\|u\|_{\mathrm{H}_{\sigma}^{1}(\mathbb{R}^{d})}=(\|u\|_{\rho}^{2}+\|\nabla u \sigma \|_{\rho}^{2})^{\frac{1}{2}}$. Finally, $(.,.) $ denotes the usual scalar product in $ L^{2}(\mathbb{R}^{d},dx)$.\\
 We also define  $\mathcal{D}:=\mathcal{C}_{c}^{\infty}([0,T])\otimes \mathcal{C}_{c}^{2}(\mathbb{R}^{d})$
the space of test functions where $\mathcal{C}_{c}^{\infty}([0,T])$ denotes the space of all real valued infinite differentiable functions with compact
 support in $[0,T]$ and $\mathcal{C}_{c}^{2}(\mathbb{R}^{d})$ the set of $C^{2}$-functions with compact support in $\mathbb{R}^{d}$.\\
We introduce $\mathcal{H}_{T}$ the space of predictable processes $(u_{t})_{t\geq 0}$ with values in $\mathrm{H}_{\rho}^{1}(\mathbb{R}^{d})$ such that
\begin{equation}\label{defH1}
\begin{split}
\|u\|_{T}=\Big(E\Big[\sup_{0\leq t\leq T}\|u_{t}\|^{2}_{\rho}\Big]+E\Big[\int_{0}^{T}\|\nabla u_{t}\sigma\|_{\rho}^{2}dt\Big]\Big)^{\frac{1}{2}}<\infty. \nonumber
\end{split}
\end{equation}
\begin{df}
We say that $u\in \mathcal{H}_{T}$ is a weak solution of the equation (\ref{SPDE}) associated with the terminal condition $\Phi$
and the coefficients $(f,g)$, if the following relation holds
almost surely, for each $\varphi\in\mathcal{D}$
\begin{eqnarray}\label{edpsss}
& &\int_{t}^{T}(u(s,.),\partial_{s}\varphi(s,.))ds+\int_{t}^{T}\mathcal{E}(u(s,.),\varphi(s,.))ds+(u(t,.),\varphi(t,.))-
(\Phi(.),\varphi(T,.))\\
&=&\int_{t}^{T}(f(s,.,u(s,.),(\nabla u\sigma)(s,.)),\varphi(s,.))ds+\sum_{i=1}^{l}\int_{t}^{T}(g(s,.,u(s,.),(\nabla u\sigma)(s,.)),\varphi(s,.))
\overleftarrow{dB_{s}^{i}},\nonumber
\end{eqnarray}
where
$ \mathcal{E}(u,\varphi)=(Lu,\varphi)=\int_{\mathbb{R}^{d}}((\nabla u\sigma)(\nabla\varphi\sigma)
+\varphi\nabla((\frac{1}{2}\sigma^{*}\nabla\sigma+b)u))(x)dx $ is
the energy  associated to the diffusion operator.
\end{df}
From Bally and Matoussi \cite{BMat}, we  have the following result:
\begin{Theorem}
Under Assumptions ${\bf(H1)-(H3)}$, there exists a unique weak solution $u \in \mathcal{H}_{T}$ of  the SPDE   \reff{SPDE}. Moreover, $u(t,x)=Y^{t,x}_t$ and $Z_t^{t,x}= \nabla u_t \sigma $,  $dt\otimes dx\otimes dP$ a.e. where $(Y_s^{t,x},Z_s^{t,x})_{t\leq s \leq T}$
is the solution of the BDSDE \eqref{BDSDE1}. Furthermore, we have for all $s \in [t,T]$, $u(s,X^{t,x}_s)=Y^{t,x}_s \textrm{ and } (\nabla u \sigma)(s,X^{t,x}_s)= Z_s^{t,x}$ $dt\otimes dx\otimes dP$ a.e.
\end{Theorem}
%\subsection{Numerical Scheme for SPDE}
\subsection{Rate of convergence for the weak solution of  SPDEs}
Our aim is to approximate the random field $(u_t(x))_{0\leq t\leq T}$ for all $x\in\mathbb{R}^d$.
We recall that the continuous approximation of the solution of BDSDE (\ref{1}) is given by:
\begin{equation}\label{2N}
Y_s^{N,t,x}:=Y_{t_{n+1}}^{N,t,x} +  \!\displaystyle{\int_{s}^{t_{n+1}}\!\!\!\!\!\!\!\!f(t_{n},\Theta_{n}^{N,t,x})du}
+\!\displaystyle{\int_{s}^{t_{n+1}}\!\!\!\!\!\!\!\!g(t_{n+1},\Theta_{n+1}^{N,t,x})\overleftarrow{dB_{u}}}
-\!\displaystyle{\int_{s}^{t_{n+1}}\!\!\!\!\!\!Z_{u}^{N,t,x} dW_u},\textrm{ } t_{n}\leq s < t_{n+1}.
\end{equation}
where
\begin{eqnarray*}
\Theta_{n}^{N,t,x}:=(X_{t_{n}}^{N,t,x},Y_{t_{n}}^{N,t,x},Z_{t_{n}}^{N,t,x}), \textrm{for all } n=0,\ldots,N.
\end{eqnarray*}
We define $n_t=\inf\{n,n=0,...N, \textrm{ such that } t\leq t_n\}\wedge N$.
We recall that the square error of the discrete time approximation is given by
\begin{eqnarray*}
Error_{N}(Y^{t,x},Z^{t,x}):=\sup_{t\leq s \leq T}E[|Y_{s}^{t,x}-Y_{s}^{N,t,x}|^{2}] +\sum_{n=n_t}^{N-1}E[\int_{t_n}^{t_{n+1}}||Z_{s}^{t,x}-Z_{s}^{N,t,x}||^{2}ds],
\end{eqnarray*}
%We introduce the continuous Euler scheme for the forward component $X$. We define for all $s\in[t_n, t_{n+1}]$
% \begin{equation*}
% X_{s}^{N,t,x}=X_{t_n}^{N,t,x}+\int_{t_n}^s b(X_{t_n}^{N,t,x})du+\int_{t_n}^s\sigma(X_{t_n}^{N,t,x})dW_{u}.
% \end{equation*}
% In this case $\{X_{s}^{N,t,x}\}_{t_n\leq s\leq t_{n+1}}$ is a solution of the standard SDE with coefficients $\widetilde{b}(x)=b(X_{t_n}^{N,t,x})$
% and $\widetilde{\sigma}(x)=\sigma(X_{t_n}^{N,t,x})$ which are piecewise constant but still Lipscitz and then $X_{s}^{N,t,x}$ satisfy
% the flow property.\\
 We recall that $u(t,x)=Y^{t,x}_t$ and $v(t,x)=Z_t^{t,x}$ $dt\otimes dx\otimes dP$ a.e. We define the process $(u_s^N,v_s^N)_{t\leq s\leq T}$,
the numerical approximation of the SPDE  \eqref{SPDE} as follows:
\begin{eqnarray}\label{numspde}
u_{s}^{N}(x):=Y_{s}^{N,s,x} \textrm{ and } v_{s}^{N}(x):=Z_{s}^{N,s,x}.
\end{eqnarray}
We define the square error between the solution of the SPDE and the numerical scheme as follows:
\begin{eqnarray}\label{Error(u,v)}
Error_{N}(u,v)&:=& \sup_{0\leq s \leq T}E[\int_{\mathbb{R}^{d}}|u_{s}^{N}(x)-u(s, x)|^{2}\rho(x)dx]\nonumber\\
&+&\sum_{n=0}^{N-1}E[\int_{\mathbb{R}^{d}}\int_{t_{n}}^{t_{n+1}}\|v_{s}^{N}(x)-v(s,x)\|^{2}ds\rho(x)dx].
\end{eqnarray}
Note that the error $Error_{N}(u,v)$ is defined by integrating over the whole domain the error $Error_{N}(Y^{t,x},Z^{t,x})$
where $(Y^{t,x},Z^{t,x})$ is the solution of the associated BDSDE.\\
The following theorem shows the convergence of the numerical scheme \reff{numspde}.
\begin{thm}
Assume that {\bf (H1)}-{\bf (H3)} hold. Then, there exists a positive constant C
(depending only on $T$, $K$, $\alpha$, $|b(0)|$, $||\sigma(0)||$, $|f(t,0,0,0)|$ and $||g(t,0,0,0)||$) such that
\begin{eqnarray}\label{ratecv-uv}
Error_{N}(u,v)\leq Ch.
\end{eqnarray}
\end{thm}
\begin{proof}
We have
\begin{eqnarray*}
E[\int_{\mathbb{R}^{d}}|u_{s}^{N}(x)-u(s, x)|^{2}\rho(x)dx]
&=&E[\int_{\mathbb{R}^{d}}|Y_{s}^{N,s,x}-Y_s^{s,x}|^{2}\rho(x)dx]\\
&\leq&\int_{\mathbb{R}^{d}}\sup_{s\leq u \leq T}E[|Y_{u}^{N,s,x}-Y_u^{s,x}|^{2}]\rho(x)dx
\end{eqnarray*}
From Theorem \ref{ratecvtheo}, we get
\begin{eqnarray*}
\sup_{0\leq s \leq T}E[\int_{\mathbb{R}^{d}}|u_{s}^{N}(x)-u(s, x)|^{2}\rho(x)dx]
\leq Ch\int_{\mathbb{R}^{d}}(1+|x|^2)\rho(x)dx
\leq Ch
\end{eqnarray*}
For the $Z$ part, we have
\begin{eqnarray*}
\sum_{n=0}^{N-1}E[\int_{\mathbb{R}^{d}}\int_{t_{n}}^{t_{n+1}}\|v_{s}^{N}(x)-v(s,x)\|^{2}ds\rho(x)dx]
=\sum_{n=0}^{N-1}E[\int_{\mathbb{R}^{d}}\int_{t_{n}}^{t_{n+1}}\|Z_{s}^{N,s,x}-Z_{s}^{s,x}\|^{2}ds\rho(x)dx].
\end{eqnarray*}
From Theorem \ref{ratecvtheo}, we get
\begin{eqnarray*}
\sum_{n=0}^{N-1}E[\int_{\mathbb{R}^{d}}\int_{t_{n}}^{t_{n+1}}\|v_{s}^{N}(x)-v(s,x)\|^{2}ds\rho(x)dx]
\leq Ch\int_{\mathbb{R}^{d}}(1+|x|^{2})\rho(x)dx\leq Ch,
\end{eqnarray*}
and then (\ref{ratecv-uv}) holds.
\end{proof}
\ep\\

\section{Implementation and numerical tests}
In this part, we are interested in implementing our numerical scheme. Our aim is only to demonstrate empirically
its convergence. We leave for future research the numerical analysis of the fully implementable algorithm.
\subsection{Notations and algorithm}
We use a path-dependent algorithm, for every fixed path of the brownian motion $B$, we approximate by a regression method the
 solution of the associated PDE. Then, we replace the conditional expectations which appear in (\ref{Yn}) and (\ref{Zn}) by
  $L^{2}(\Omega,\mathcal{P})$ projections on the function basis approximating $L^{2}(\Omega,\mathcal{F}_{t_{n}})$.
   We compute $Z_{t_{n}}^{N}$ in an explicit manner  and  $Y_{t_{n}}^{N}$ in a implicit way by using I Picard iterations where
   I is a natural number. Actually, we proceed as in \cite{gob}, except that in our case the solutions $Y_{t_{n}}^{N}$ and $Z_{t_{n}}^{N}$  are measurable functions of $(X^{N}_{t_{n}},(\Delta B_{i})_{n\leq i \leq N-1} )$. So, each solution given by our algorithm depends on the fixed path of B.

\subsubsection{Numerical scheme}
We take $k=d=1$ i.e. $W$ and $B$ are one dimensional Brownian motions.
For each fixed path of $B$, the solution of (\ref{forward})-(\ref{1}) is approximated by $(Y^{N},Z^{N})$ defined by (\ref{Yn})-(\ref{Zn})\\
%the following algorithm.\\
%For $0\leq n\leq N-1$:\\
%\begin{equation}\label{2}
%Y_{t_{n}}^{N} = E_{t_{n}}\Big[Y_{t_{n+1}}^{N} + hf(X_{t_{n}}^{N},Y_{t_{n}}^{N},Z_{t_{n}}^{N})
%+g(X_{t_{n+1}}^{N},Y_{t_{n+1}}^{N},Z_{t_{n+1}}^{N})
%\Delta B_{n}\Big],
%\end{equation}
%\begin{equation}\label{3}
%h Z_{t_{n}}^{N} = E_{t_{n}}\Big[Y_{t_{n+1}}^{N}\Delta W_{n}+ g(X_{t_{n+1}}^{N},Y_{t_{n+1}}^{N},Z_{t_{n+1}}^{N})
%\Delta B_{n}\Delta W_{n}\Big].
%\end{equation}
We stress that at each discretization time, the solution of the algorithm depends on the fixed path of the brownian motion $B$.
\subsubsection{Vector spaces of functions}
At every $t_{n}$, we select 2 deterministic functions bases
$(p_{i,n}(.))_{i\in\{0,1\}}$ and we look for approximations of $Y_{t_{n}}^{N}$ and $Z_{t_{n}}^{N}$
which will be denoted respectively by $y_{n}^{N}$ and $z_{n}^{N}$, in the vector space $(P_{i,n}(.))_{i\in\{0,1\}}$
 spanned by the basis
$p_{0,n}(.)$ and $p_{1,n}(.)$.
 Each basis $p_{i,n}(.)$ is considered as a vector of functions of dimension $L_{i,n}$.
 % The vector space of functions spanned by $p_{i,k}(.)$ is denoted by $P_{i,k}(.)$.
In other words, $P_{i,n}(.)=\{ \alpha .p_{i,n}(.), \alpha \in \mathbb{R} ^{L_{i,n}} \}$.\\
As an example, we cite the hypercube basis $(\textbf{HC})$ used in \cite{gob}. In this case, $p_{i,n}(.)$ does not depend nor on $i$ neither on $n$ and its dimension is simply denoted by $L$.
 A domain $D\!\subset\!\mathbb{R}$ centered on $X_{0}\!=\!x$, that is $D=(x-a,x+a]$, can be partitioned on small hypercubes of edge
  $\delta$. Then, $D\!=\!\bigcup_{i_{1},\ldots,i_{d}}\!D_{i_{1,\ldots,i_{d}}}$ where
   $D_{i_{1},\ldots,i_{d}}\!=\!(x-a+i_{1}\delta,x-a+(i_{1}+1)\delta]\times\ldots\times (x-a+i_{d}\delta,x-a+(i_{d}+1)\delta]$.
  Finally we define $p_{i,n}(.)$ as the indicator functions of this set of hypercubes.

\subsubsection{Monte Carlo simulations}
To compute the projection coefficients $\alpha$, we will use M independent Monte Carlo simulations
of $X\!_{\!t_{n}\!}\!^{N}$ and $\!\Delta W_{\!n\!}$ which will be respectively denoted by
 $X\!_{\!t_{n}\!}^{N,m}$ and $\Delta \!W_{n}^{m}$, $\!m\!=\!1\!,\ldots,\!M$.

\subsubsection{Description of the algorithm}
$\rightarrow$ Initialization: For $n=N$, we set $(y_{N}^{N,m,I})=(\Phi(X^{N,m}_{t_{N}}))$ and $(z_{N}^{N,m})=0$ .\\
$\rightarrow$ Iteration: For $n=N-1,\ldots,0$:\\
$\bullet$ We approximate (\ref{Zn}) by computing
%for all $j_{1} \in \{1,\ldots,k\}$ and $j_{2} \in \{1,\ldots,d\}$
\begin{eqnarray*}
\alpha^{M}_{1,n}&=&\mathop{\rm arginf}\limits_ {\alpha} \frac{1}{M}\sum_{m=1}^{M}\!\Big|y_{n+1}^{N,M,I}(\!X^{N,m}_{t_{n+1}})\!\frac{\!\Delta W_{n}^{m}\!}{h}\\
&\!+\!&\!g\Big(\!X^{N,m}_{t_{n+1}}\!,\!y_{n+1}^{N,M,I}\!(\!X^{N,m}_{t_{n+1}}\!),z_{n+1}^{N,M}\!(\!X^{N,m}_{t_{n+1}}\!)\!\Big)\!\frac{\Delta B_{n}\Delta \!W_{n}^{m}}{h}
-\alpha.p_{1,n}(X_{t_n}^{N,M})\Big|^{2}.
\end{eqnarray*}
Then we set $z_{n}^{N,M}(.)=(\alpha^{M}_{1,n}.p_{1,n}(.))$. \\
$\bullet$ We use $I$ Picard iterations to obtain an approximation of $Y_{t_{n}}$ in (\ref{Yn}):\\
$\cdot$ For $i=0$: $\alpha^{M,0}_{0,n}=0$.\\
$\cdot$ For $i=1,\ldots,I$: We approximate (\ref{Yn}) by calculating $\alpha^{M,i}_{0,n}$ as the minimizer of:
\begin{eqnarray*}
\!\frac{1}{M}\sum_{m=1}^{M}\!\Big|y_{n+1}^{N,M,I}(X^{N,m}_{t_{n+1}})
\!+hf\!\Big(\!X^{N,m}_{t_{n}},\!y_{n}^{N,M,i-1}(X^{N,m}_{t_{n}}\!)\!,\!z_{n}^{N,M}\!(\!X^{N,m}_{t_{n}}\!)\!\Big)\\
+g\!\Big(\!X^{N,m}_{t_{n+1}}\!,\!y_{n+1}^{N,M,I}\!(\!X^{N,m}_{t_{n+1}}\!)\!,\!z^{N,M}_{n+1}\!(\!X^{N,m}_{t_{n+1}}\!)\!\Big)\!\Delta B_{n}-\!\alpha.p_{0,n}(X_{t_n}^{N,M})\Big|^{2}.
\end{eqnarray*}
Finally, we define $y_{n}^{N,M,I}(.)$ as:
\begin{eqnarray*}
y_{n}^{N,M,I}(.)=(\alpha^{M,I}_{0,n}.p_{0,n}(.)).
\end{eqnarray*}
%\subsection{One-dimensional case (Case when $d=k=l=1$)}
\subsubsection{Function bases}
%\par
We use the basis ($\textbf{HC}$) defined above. So we set:
\begin{eqnarray*}
d_{1}=\min_{n,m}X^{m}_{t_{n}},\quad d_{2}=\max_{n,m}X^{m}_{t_{n}} \textrm{ and } L=\frac{d_{2}-d_{1}}{\delta}
\end{eqnarray*}
where $\delta$ is the edge of the hypercubes $(D_{j})_{1\leq j\leq L}$ defined by $D_{j}=\Big[d+(j-1)\delta,d+j\delta\Big),\forall j$.\\
At each time $t_{n}$, we set
\begin{eqnarray*}
1_{D_{j}}(X^{N,m}_{t_{n}})=1_{[d+(j-1)\delta,d+j\delta)} (X^{N,m}_{t_{n}}),j=1,\ldots,L
\end{eqnarray*}
and
\begin{eqnarray*}
\!(\!p^{m}_{i,n}(.)\!)\!=\!\Big\{\!\sqrt{\frac{M}{card(D_{j})}}\!1_{D_{j}}\!(\!X^{N,m}_{t_{n}}\!)\!,\!1\! \leq \!j\!\leq\! L\Big\},i=0,1,
\end{eqnarray*}
where $Card(D_{j})$ denotes the number of simulations of $X^{N}_{t_{n}}$ which are in the cube $D_{j}$.\\
This system is orthonormal with respect to the empirical scalar product defined by
\begin{eqnarray*}
<\psi_{1},\psi_{2}>_{n,M}:=\frac{1}{M}\!\sum_{m=1}^{M}\!\psi_{1}\!(\!X^{N,m}_{t_{n}}\!) \psi_{2}\!(X^{N,m}_{t_{n}}\!).
\end{eqnarray*}
In this case, the solutions of our least squares problems are given by:
\begin{eqnarray*}
\alpha^{M}_{1,n}&=&\frac{1}{M}\sum_{m=1}^{M} p_{1,n}(X^{N,m}_{t_{n}}) \Big\{ y_{n+1}^{N,M,I}(X^{N,m}_{t_{n+1}}) \frac{\Delta W^{m}_{n}}{h}\\
&+&g\Big(X^{N,m}_{t_{n+1}},y_{n+1}^{N,M,I}(X^{N,m}_{t_{n+1}}),z^{N,M,}_{n+1}(X^{N,m}_{t_{n+1}})\Big)\frac{\Delta B_{n}\Delta W^{m}_{n}}{h}\Big\},\\
\alpha^{M,i}_{0,n}&=&\frac{1}{M}\sum_{m=1}^{M} p_{0,n}(X^{N,m}_{t_{n}}) \Big\{y_{n+1}^{N,M,I}(X^{N,m}_{t_{n+1}})+ h f\Big(X_{t_{n}}^{N,m},y_{n}^{N,M,i-1}(X^{N,m}_{t_{n}}),z_{n}^{N,M}(X^{N,m}_{t_{n}})\Big)\\
&+&g\Big(X^{N,m}_{t_{n+1}},y_{n+1}^{N,M,I}(X^{N,m}_{t_{n+1}}),z^{N,M}_{n+1}(X^{N,m}_{t_{n+1}})\Big)\Delta B_{n}\Big\}.
\end{eqnarray*}
\begin{Remark}
We note that for each value of $M$, $N$ and $\delta$, we launch the algorithm $50$ times and we denote by $(Y_{0,m'}^{0,x,N,M,I})_{1\leq m' \leq50}$ the set of collected values. Then we calculate the empirical mean $\overline{Y}_{0}^{0,x,N,M,I}$ and the empirical standard deviation $\sigma^{N,M,I}$defined by:
\begin{equation}\label{approxY0}
\overline{Y}_{0}^{0,x,N,M,I}\!=\frac{1}{50}\sum_{m'=1}^{50}\!Y_{0,m'}^{0,x,N,M,I} \textrm{ and } \sigma^{N,M,I}\!=\!\sqrt{\frac{1}{49}\sum_{m'=1}^{50}|Y_{0,m'}^{0,x,N,M,I}\!-\!\overline{Y}_{0}^{0,x,N,M,I}|^{2}}.
\end{equation}
We also note before starting the numerical examples that our algorithm converges after at most three Picard iterations.
Finally, we stress that (\ref{approxY0}) gives us an approximation of $u(0,x)$ the solution of the SPDE (\ref{SPDE}) at time
 $t=0$ given the path of $B$.
\end{Remark}
\subsection{Examples}
\subsubsection{Case when $f$ and $g$ are linear in $y$ and independent of $z$}
\begin{eqnarray*}
\begin{cases}
&dX_{t}=X_{t}(\mu dt+\sigma dW_t),\\
&\Phi(x)=-x+K,\textrm{ }f(y)=a_{0}y,\textrm{ } g(y)=b_{0}y
\end{cases}
\end{eqnarray*}
and we set $K=115$, $r=0.01$, $R=0.06$, $X_{0}=100$, $\mu=0.05$, $\sigma=0.2$, $T=0.25$, $d_{1}=60$, $d_{2}=200$, $a_{0}$ and $b_{0}$ are fixed constants.\\
Let $Y_{explicit}$ be the solution of our BDSDE in this particular case. By the integration by parts formula, we get
\begin{eqnarray*}
Y_{t,explicit}^{t,x}=E[\Phi(X^{t,x}_{T})e^{a_{0}(T-t)+b_{0}(B_{T}-B_{t})-\frac{1}{2}b_{0}^{2}(T-t)}/\mathcal{F}_{t,T}^{B}].
\end{eqnarray*}
At t=0, we have
\begin{eqnarray*}
Y_{0,explicit}^{0,x}&=&E[\Phi(X^{0,x}_{T})e^{(a_{0}-\frac{1}{2}b_{0}^{2})T+b_{0}B_{T}}/\mathcal{F}_{0,T}^{B}]\\
&=&e^{(a_{0}-\frac{1}{2}b_{0}^{2})T+b_{0}B_{T}}E[\Phi(X^{0,x}_{T})]\\
&=&e^{(a_{0}-\frac{1}{2}b_{0}^{2})T+b_{0}B_{T}}(K-xe^{\mu T}).
\end{eqnarray*}
Then, we define $\overline{Y}_{0}^{0,x,N,M,I}$ as the numerical approximation of the solution of the BDSDE in this case
 (computed by our algorithm) and $\sigma^{N,M,I}$ as its standard deviation.
 %In the other hand, we compute the solution $Y_{0,explicit}^{0,x}$ in this linear case by using the explicit formula of the expectation of $\Phi(X_T^{0,x})$, as follows\\

%\begin{eqnarray*}
%Y_{explicit}^{0,x}=e^{(a_{0}-\frac{1}{2}b_{0}^{2})T+b_{0}B_{T}}E[\Phi(X^{0,x}_{T})]=e^{(a_{0}-\frac{1}{2}b_{0}^{2})T+b_{0}B_{T}}(K-xe^{\mu T}).\\
%%&=&e^{(a_{0}-\frac{1}{2}b_{0}^{2})T+b_{0}B_{T}}\{ E[(X^{0,x}_{T}-K_{1})^{+}]-2E[(X^{0,x}_{T}-K_{2})^{+}] \} \\
%%&=&e^{(a_{0}-\frac{1}{2}b_{0}^{2})T+b_{0}B_{T}}\{xe^{\mu}N(d_{1,K_{1}})-K_{1}N(d_{2,K_{1}})-2[ %xe^{\mu}N(d_{1,K_{2}})-K_{2}N(d_{2,K_{2}}) ]\},
%\end{eqnarray*}
%where $N(.)$ denotes the standard Gaussian cumulative distribution function and
%\begin{eqnarray*}
%\begin{cases}
%&d_{1,R}=\frac{1}{\sigma \sqrt{T}}\{ \ln{\frac{x}{R}}+(\mu+\frac{\sigma^{2}}{2})T \},\\
%&d_{2,R}=d_{1,R}-\sigma \sqrt{T},
%\end{cases}
%\end{eqnarray*}
%for a real non negative constant $R$.\\

For $a_{0}=0.5$, $b_0=0.5$ and  $\delta=1$
\begin{center}
N=20, $Y_{explicit}^{0,x}=13.724$
\begin{tabular}{|r|c|c|c|c|}

\hline
$M$ & $\overline{Y}_{0}^{0,x,N,M,I}(\sigma^{N,M,I})$ & $\frac{|Y_{explicit}^{0,x}-\overline{Y}_{0}^{0,x,N,M,I}|}{Y_{explicit}^{0,x}}$ \\

\hline\
   100& 13.911(1.178) & 0.013\\

\hline
  1000& 13.793(0.309) & 0.004\\

\hline
  5000& 13.848(0.117) & 0.009\\

\hline
 10000& 13.856(0.091) & 0.009\\

\hline

\end{tabular}
\\
\end{center}

For $a_{0}=0.5$, $b_{0}=0.5$ and  $\delta=0.5$
\begin{center}
N=30, $Y_{explicit}^{0,x}=14.115$
\begin{tabular}{|r|c|c|c|c|}

\hline
$M$ & $\overline{Y}_{0}^{0,x,N,M,I}(\sigma^{N,M,I})$ & $\frac{|Y_{explicit}^{0,x}-\overline{Y}_{0}^{0,x,N,M,I}|}{Y_{explicit}^{0,x}}$ \\

\hline\
   100& 14.245(1.045) & 0.009  \\

\hline
  1000& 14.194(0.337) & 0.005\\

\hline
  5000& 14.235(0.129) & 0.008\\

\hline
  10000& 14.263(0.101) &0.01\\
\hline
\end{tabular}

\end{center}
%For $a_{0}=0.5$, $b=3$ and  $\delta=1$
%\begin{center}
%N=20, $Y_{BS}^{0,x}=0.581$
%\begin{tabular}{|r|c|c|c|c|}
%
%\hline
%$M$ & $\overline{Y}_{0}^{0,x,N,M}(\sigma^{N,M})$ & $|Y_{BS}^{0,x}-\overline{Y}_{0}^{0,x,N,M}|$ \\
%
%\hline\
%   100& 0.730(0.089) & 0.149 \\
%
%\hline
%  1000& 0.684(0.031) & 0.103   \\
%
%\hline
%
% $10^{4}$& 0.688(0.012) & 0.107   \\
%
%\hline
%
%
%\end{tabular}
%
%\end{center}
%For $a_{0}=0.5$, $b=3$ and  $\delta=0.5$
%\begin{center}
%N=30, $Y_{BS}^{0,x}=0.687$
%\begin{tabular}{|r|c|c|c|c|}
%
%\hline
%$M$ & $\overline{Y}_{0}^{0,x,N,M}(\sigma^{N,M})$ & $|Y_{BS}^{0,x}-\overline{Y}_{0}^{0,x,N,M}|$ \\
%
%\hline\
%   100& 0.9(0.134) & 0.212  \\
%
%\hline
%  1000& 0.86(0.043) & 0.173 \\
%
%\hline
%
%  $10^{4}$& 0.871(0.016) & 0.183  \\
%
%\hline
%
%\end{tabular}
%
%\end{center}
In the linear case we have a benchmark. We see that in the maturity the numerical approximation of the
 BDSDE's solution is closed to the exact solution. We also note that the bias is constant depending on the number of simulation.

\subsubsection{Comparison of numerical approximations of the solutions of the FBDSDE and the FBSDE: the general case}
Now we set
\begin{eqnarray*}
\begin{cases}
%&dX_{t}=X_{t}(\mu dt+\sigma dWt),\\
&\Phi(x)=-x+K,\\
&f(t,x,y,z)=-\theta z-ry+(y-\frac{z}{\sigma})^{-}(R-r),\\
&g(t,x,y,z)=0.1z+0.5y+log(x)
\end{cases}
\end{eqnarray*}
The associated nonlinear SPDE is given by:
\begin{equation*}
\begin{split}
 du_t (x) +  \big(Lu_t (x)
 \;   + &  f(t,x,u_t (x),\nabla u_t \sigma  (x)) \big) \, dt  +  g(t,x,u_t(x),\nabla u_t \sigma (x))\cdot \overleftarrow{dB}_t  = 0, \,\
\end{split}
\end{equation*}
where
\begin{eqnarray*}
Lu_t (x)=\sigma^2 x^2\frac{\partial^2}{\partial x^2}u_t(x)+\mu x\frac{\partial}{\partial x}u_t(x).
\end{eqnarray*}
We set $\theta=(\mu-r)/\sigma$, $K=115$, $X_{0}=100$, $\mu=0.05$, $\sigma=0.2$, $r=0.01$, $R=0.06$, $\delta=1$, $N=20$, $T=0.25$
 and we fix $d_{1}=60$ and $d_{2}=200$ as in \cite{gobet1}.
  %The functions $g_1$,$g_2$ and $g_3$ taken in the following are examples
% of the function g. They are sufficiently regular and Lipschitz on $[60,200]\times\mathbb{R}\times\mathbb{R}$ and could be extended to regular Lipschitz functions on $\mathbb{R}^3$. In this case, Assumptions
%(\textbf{H1})-(\textbf{H3}) are satisfied.\\
The function g is sufficiently regular and Lipschitz on $[60,200]\times\mathbb{R}\times\mathbb{R}$ and could be extended to regular
 Lipschitz function on $\mathbb{R}^3$. In this case, Assumptions (\textbf{H1}), (\textbf{H2}) and (\textbf{H3})(\textbf{i})
 are satisfied. (\textbf{H3})(\textbf{ii}) is not satisfied because $f$ is not differentiable.\\
 %as  $x \in [d_{1},d_{2}]=[60,200]$. They are also satisfied for the functions $g_{2}$ and $g_{3}$ taken below as  $x \in %[60,200]$.\\
We compare the numerical solution of our BDSDE (noted again $\overline{Y}_{t}^{t,x,N,M,I}=u_t(X_0)$)  and the BSDE's one (noted here by $\overline{Y}_{t,BSDE}^{0,x,N,M}$ ), without $g$ and $B$.\\
%We finally note before starting the numerical examples that, for the contraction constant taken in the following
%($\alpha=0.1$), our algorithm converges after at most three Picard iterations. We stress also that the additional assumption $0\leq \alpha^2 d<1$ is satisfied since we are working in the one-dimensional case.\\
When $t$ is close to maturity
%\begin{center}
%\begin{tabular}{|r|c|c|c|}
%
%\hline
%$M$ & $\overline{Y}_{t_{19},BSDE}^{0,x,N,M}(\sigma^{N,M})$ & $\overline{Y}_{t_{19}}^{0,x,N,M,I}(\sigma^{N,M,I})$ \\
%
%\hline\
%  128& 13.748(0.879)& 15.452(0.948)\\
%
%\hline
%  512& 13.827(0.384)& 15.534(0.409)\\
%
%\hline
%
%  2048& 13.762(0.223)& 15.464(0.240)\\
%
%\hline
%
%  8192& 13.781(0.091)& 15.484(0.097)\\
%
%\hline
%
%  32768& 13.796(0.054)& 15.501(0.058)\\
%
%\hline
%
%\end{tabular}
%
%\end{center}

\begin{center}
\begin{tabular}{|r|c|c|c|}

\hline
$M$ & $\overline{Y}_{t_{15},BSDE}^{0,x,N,M}(\sigma^{N,M})$ & $u_{t_{15}}(X_0)=\overline{Y}_{t_{15}}^{0,x,N,M,I}(\sigma^{N,M,I})$ \\

\hline\
  128& 14.168(0.905)& 17.894(1.096) \\

\hline
  512& 14.113(0.388)& 17.774(0.429) \\

\hline

  2048& 13.988(0.226)& 17.607(0.270)\\

\hline

  8192& 13.985(0.093)& 17.623(0.104)\\

\hline

  32768& 13.994(0.055)& 17.627(0.064)\\

\hline

\end{tabular}

\end{center}

When $t=0$
\begin{center}
\begin{tabular}{|r|c|c|c|}

\hline
$M$ & $\overline{Y}_{0,BSDE}^{0,x,N,M}(\sigma^{N,M})$ & $u_{0}(X_0)=\overline{Y}_{0}^{0,x,N,M,I}(\sigma^{N,M,I})$ \\

\hline\
  128& 15.431(1.005)& 13.571(1.146)
 \\

\hline
  512& 15.029(0.428)& 13.173(0.500) \\

\hline

  2048& 14.763(0.243)& 12.885(0.280)\\

\hline

  8192& 14.718(0.098)& 12.825(0.106)\\

\hline

  32768& 14.715(0.060)& 12.804(0.064)\\

\hline

\end{tabular}

\end{center}
  We see the convergence of the BDSDE's solution when we increase the number of simulations $M$.\\
  %%%%%%%%%%%%%%%%%%%%%%%%%%%%%%%%%%%%%%%%%%%%%%%%
%, we observe that when $g$ doesn't depend on $z$, the standard deviation of our solution is smaller, either we are close to maturity or at $t_{0}$. However, for the BDSDE's solution value, when we are close to maturity (at $t_{15}$), it is closer to the BSDE's one when there is dependence of the function $g_1$ on $z$. At $t_{0}$, the opposite becomes true.\\

In figure \ref{figure1}, we examine the convergence of our scheme for five different path of the Brownian $B$.
 We fix all the parameters
($\delta=1 \textrm{ and }M=2000$ ) and we draw the map of the BDSDE's solution with respect to the number of time discretization
 steps $N$.
\begin{center}
\begin{figure}[!h]
\includegraphics[width=15.5cm]{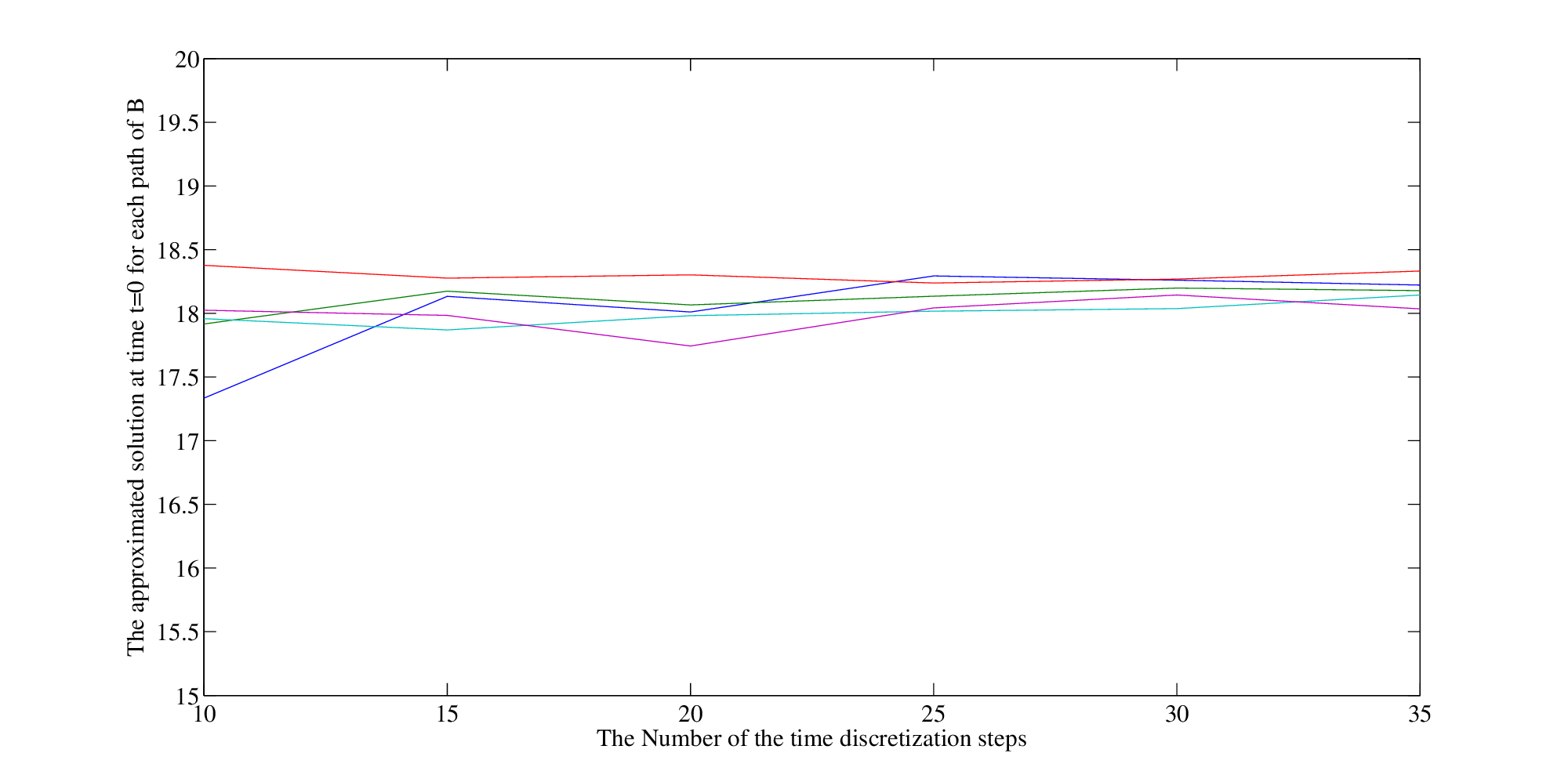}
\caption{The BDSDE's solution with respect to the number of time discretization steps for five different paths of B. The figure is obtained for $M=2000$ and $\delta=1$. }\label{figure1}
\end{figure}
\end{center}
%
%In the first figure, we are interested in analyzing the dependence of the BDSDE's solution on the variable $z$ in the function $g$. So, we variate the parameters $N$, $M$ and $\delta$ as above and we draw the maps of BDSDE's solution at $t=0$ with respect to j.
%\begin{center}
%\begin{figure}[!h]
%\includegraphics[width=15.5cm]{BDSDEj.eps}
%\caption{Simultaneous variation of parameters ($N$,$\delta$,$M$): The BDSDE's solution with repect to the parameter j is with star markers. Confidence intervals are with dotted lines. }\label{figure2}
%\end{figure}
%\end{center}
\newpage
We see on Figure \ref{figure2} the impact of the function $g$ on the solution; we variate $N$, $M$ and $\delta$ as in \cite{gob},
 by taking these quantities as follows: First we fix $d_{1}=40$ and $d_{2}=180$ (which means that  $x \in [d_{1},d_{2}]=[40,180]$
  and in this case our assumptions (\textbf{H1})-(\textbf{H3}) are satisfied). Let $j \in \mathbb{N}$,
  we take $N=2(\sqrt{2})^{(j-1)}$, $M=2(\sqrt{2})^{3(j-1)}$ and $\delta=50/(\sqrt{2})^{(j-1)}$.
  Then, we draw the map of each solution at $t=0$ with respect to j.
\begin{center}
\begin{figure}[!h]
\includegraphics[width=15.5cm]{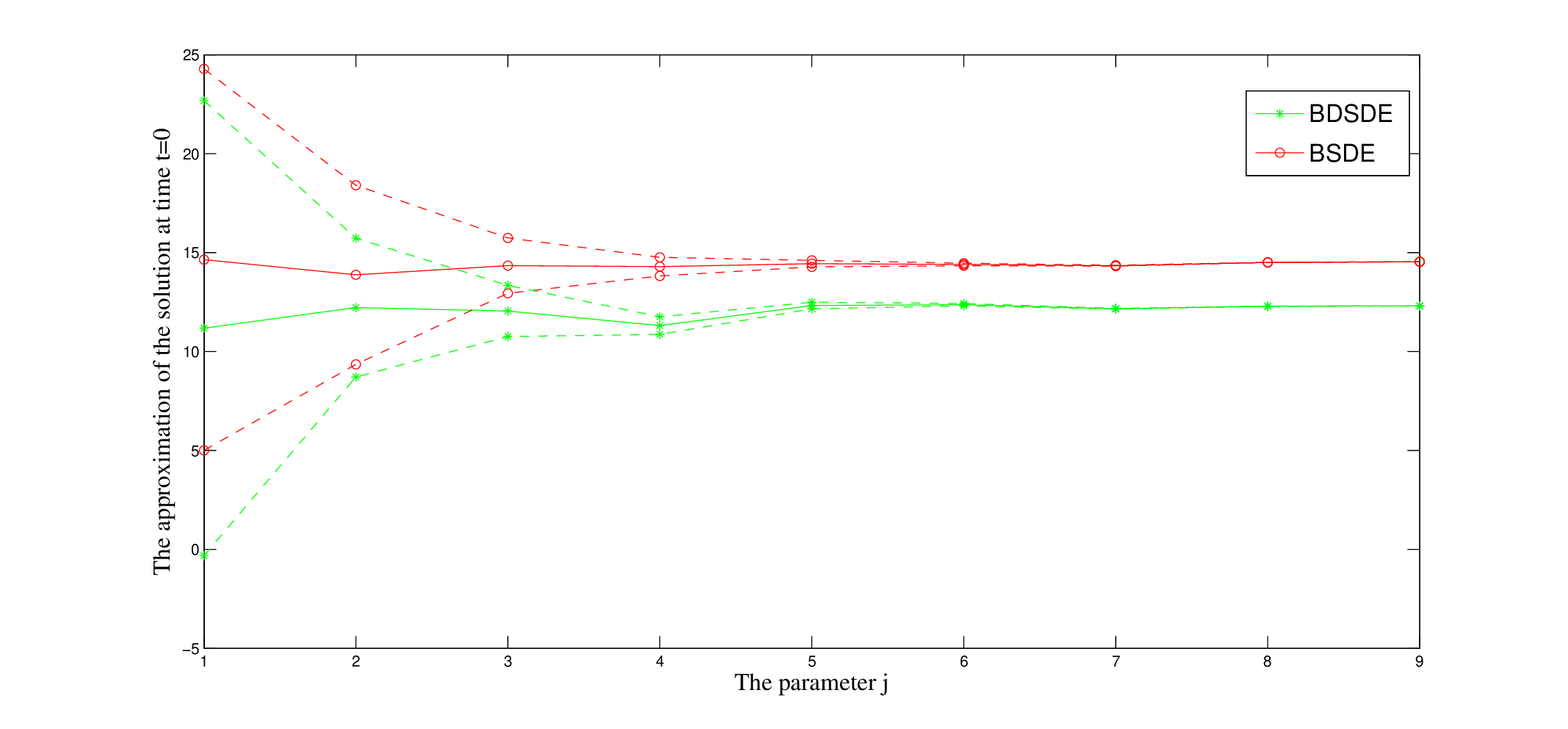}
\caption{Comparison of the BSDE's solution and the BDSDE's one: The solution of the BSDE is with circle markers,
the solution of the BDSDE is with star markers. Confidence intervals are with dotted lines.}\label{figure2}
\end{figure}
\end{center}
\newpage
\section{Appendix}
\setcounter{equation}{0}
\setcounter{Assumption}{0}
\setcounter{Example}{0}
\setcounter{Theorem}{0}
\setcounter{Proposition}{0}
\setcounter{Corollary}{0}
\setcounter{Lemma}{0}
\setcounter{Definition}{0}
\setcounter{Remark}{0}
\subsection{{\bf Proof of Lemma \ref{first step}.}}
From (\ref{1N}), we have for all $t \in [t_{n},t_{n+1})$
\begin{eqnarray*}
\delta Y_{t}^{N} =\delta Y_{t_{n+1}}^{N}+ \int_{t}^{t_{n+1}} \delta f_{s}ds
+\int_{t}^{t_{n+1}} \delta g_{s} \overleftarrow{dB_{s}} - \int_{t}^{t_{n+1}} \delta Z_{s}^{N}dW_{s}.
\end{eqnarray*}
Using the Generalized It\^{o}'s Lemma (see Lemma 1.3, \cite{par94}), we obtain
\begin{eqnarray*}
|\delta Y_{t}^{N}|^{2} + \int_{t}^{t_{n+1}} \|\delta Z_{s}^{N}\|^{2}ds -  |\delta Y_{t_{n+1}}^{N}|^{2}&=&2\int_{t}^{t_{n+1}}\!\!\!\!\!\! (\delta Y_{s}^{N},\delta f_{s})ds + 2\int_{t}^{t_{n+1}} \!\!\!\!\!\!\! (\delta Y_{s}^{N},\delta g_{s} \overleftarrow{dB_{s}})\\
 &+& \!\int_{t}^{t_{n+1}} \!\!\!\!\|\delta g_{s}\|^{2}ds\!-\! 2\!\!\int_{t}^{t_{n+1}}\!\!\!\!\!\! (\!\delta Y_{s}^{N}\!,\!\delta Z_{s}^{N}dW_{s}),\!\forall t \!\in\! [t_{n},t_{n+1}),
\end{eqnarray*}
where $(.,.)$ is the inner product associated with the euclidean norm.\\
Then taking the expectation, we have
\begin{eqnarray}\label{defAt}
A^n_{t} := E[|\delta Y_{t}^{N}|^{2}]+\int_{t}^{t_{n+1}}E[\|\delta Z_{s}^{N}\|^{2}]ds-E[|\delta Y_{t_{n+1}}^{N}|^{2}]&=&2\int_{t}^{t_{n+1}} E[(\delta Y_{s}^{N},\delta f_{s})]ds\nonumber\\
&+&\int_{t}^{t_{n+1}} E[\|\delta g_{s}\|^{2}]ds.
\end{eqnarray}
%%%%%%%%%%%%%%%%%%%%%%%%%%%%%%%%%%%%%%%%%%%%%%%%%%%%%%%%%%%%%%%%%%%%%%%%%%%%%%%%%%%%%%%%%%%%%%%
From Assumption \textbf{(H2)-(ii)}, we have
\begin{eqnarray}\label{maj33-1}
&&\int_{t}^{t_{n+1}} E[\|\delta g_{s}\|^{2}]ds \leq K^{2}h^{2} + K^{2}\int_{t}^{t_{n+1}} E[|X_{s}-X_{t_{n+1}}^{N}|^{2}]ds\nonumber\\
&+&K^{2} \int_{t}^{t_{n+1}}E[|Y_{s}-Y_{t_{n+1}}^{N}|^{2}]ds
+\alpha^{2}E\big[\int_{t}^{t_{n+1}}||Z_{s}-Z_{t_{n+1}}^{N}||^{2}ds\big].
\end{eqnarray}
Using the Young's inequality, for a positive constant $\epsilon$, we obtain for all $n=0,\ldots,N-1$,
\begin{eqnarray}\label{eq1}
E\big[\int_{t}^{t_{n+1}}||Z_{s}-Z_{t_{n+1}}^{N}||^{2}ds\big]&\leq& (1+\frac{1}{\epsilon})E\big[\int_{t}^{t_{n+1}}||Z_{s}-\bar{Z}_{t_{n+1}}||^{2}ds\big]\nonumber\\
&+&(1+\epsilon)E\big[\int_{t}^{t_{n+1}}||\bar{Z}_{t_{n+1}}-Z_{t_{n+1}}^{N}||^{2}ds\big].
\end{eqnarray}
%Let us note that the last inequality is not needed for $n=N-1$ , since $Z_{t_{N}}^{N}=0$ and then $\displaystyle{E\big[\int_{t}^{t_{N}}||Z_{s}-Z_{t_{N}}^{N}||^{2}ds\big]=E\big[\int_{t}^{t_{N}}||Z_{s}||^{2}ds\big]},\forall t \in [t_{N-1},t_N)$.\\
For all $n=0,\ldots,N-2$, we use Lemma \ref{lemme relation Z_set Z_tn}, the definition of $\bar{Z}$ and the Jensen's inequality to get
\begin{eqnarray*}
E\big[||\bar{Z}_{t_{n+1}}-Z_{t_{n+1}}^{N}||^{2}\big]&=& E\Big[||\frac{1}{h}E_{t_{n+1}}\big[\int_{t_{n+1}}^{t_{n+2}}\delta Z_r^N dr\big]||^2\Big].\nonumber\\
&\leq&\frac{1}{h^2}E\Big[E_{t_{n+1}}\big[||\int_{t_{n+1}}^{t_{n+2}}\delta Z_r^N dr||^2\big]\Big].
\end{eqnarray*}
By using Cauchy Schwartz inequality, we obtain for all $n=0,\ldots,N-2$
\begin{eqnarray}\label{intercale-Z-bar}
E\big[||\bar{Z}_{t_{n+1}}-Z_{t_{n+1}}^{N}||^{2}\big]
\leq\frac{1}{h}E\Big[\int_{t_{n+1}}^{t_{n+2}}\|\delta Z_r^N\|^2 dr\Big].
\end{eqnarray}
Plugging (\ref{intercale-Z-bar}) in (\ref{eq1}) then (\ref{eq1}) in (\ref{maj33-1}), we get for all $n=0,\ldots,N-1$
\begin{eqnarray}\label{maj34-1}
&&\int_{t}^{t_{n+1}} E[\|\delta g_{s}\|^{2}]ds \leq K^{2}h^{2} + K^{2}\int_{t}^{t_{n+1}} E[|X_{s}-X_{t_{n+1}}^{N}|^{2}]ds
+K^{2} \int_{t}^{t_{n+1}}E[|Y_{s}-Y_{t_{n+1}}^{N}|^{2}]ds\nonumber\\
&+& (1+\frac{1}{\epsilon})\alpha^{2}\int_{t}^{t_{n+1}}E[||Z_{s}-\bar{Z}_{t_{n+1}}||^{2}]ds
+(1+\epsilon)\alpha^{2}\mathds{1}_{\{n<N-1\}}\int_{t_{n+1}}^{t_{n+2}}E[\|\delta Z_s^N\|^2]ds.
\end{eqnarray}
%The previous inequality becomes trivially for $n=N-1$,
%\begin{eqnarray}\label{maj34-N-1}
%&&\int_{t}^{t_{N}} E[\|\delta g_{s}\|^{2}]ds \leq K^{2}h^{2} + K^{2}\int_{t}^{t_{N}} E[|X_{s}-X_{t_{N}}^{N}|^{2}]ds
%+K^{2} \int_{t}^{t_{N}} E[|Y_{s}-Y_{t_{N}}^{N}|^{2}]ds\nonumber\\
%&+& \alpha^{2}E\big[\int_{t_{N-1}}^{t_{N}}||Z_{s}||^{2}ds\big],\forall t \in [t_{N-1},t_N).
%\end{eqnarray}
%We mention that in the rest of the proof, we will omit to treat the case $n=N-1$. This case remains only to replace inequality
%$(\ref{maj34-1})$ by inequality $(\ref{maj34-N-1})$ in the following estimations, which is simpler to handle.\\
We set $\alpha':=(1+\epsilon)\alpha^{2}$. We choose $\epsilon$ such that $\alpha' \in (0,1)$. This is possible since $\alpha^{2} \in (0,1)$.
 Then, we use the inequality $2ab \leq \frac{1-\alpha'}{16K^{2}}a^{2} + \frac{16K^{2}}{1-\alpha'}b^{2}$ and equation (\ref{maj34-1}) to obtain for all $n=0,\ldots,N-1$
\begin{eqnarray}\label{}
A^n_{t} &\leq &\frac{16K^{2}}{1-\alpha'}\int_{t}^{t_{n+1}} E[|\delta Y_{s}^{N}|^{2}]ds + \frac{1-\alpha'}{16K^{2}}\int_{t}^{t_{n+1}} E[|\delta f_{s}|^{2}]ds + K^{2}h^{2} \nonumber\\
&+& K^{2}\int_{t}^{t_{n+1}} E[|X_{s}-X_{t_{n+1}}^{N}|^{2}]ds + K^{2} \int_{t}^{t_{n+1}}E[|Y_{s}-Y_{t_{n+1}}^{N}|^{2}]ds\nonumber\\
&+&(1+\frac{1}{\epsilon})\alpha^{2}\int_{t}^{t_{n+1}}E[||Z_{s}-\bar{Z}_{t_{n+1}}||^{2}]ds
+\alpha'\mathds{1}_{\{n<N-1\}}\int_{t_{n+1}}^{t_{n+2}}E[\|\delta Z_s^N\|^2]ds\nonumber
\end{eqnarray}
Now using Assumption \textbf{(H2)-(i)} in the last inequality, we get
\begin{eqnarray}\label{}
A^n_{t} &\leq & \frac{16K^{2}}{1-\alpha'}\int_{t}^{t_{n+1}} E[|\delta Y_{s}^{N}|^{2}]ds + \frac{1-\alpha'}{16K^{2}} 4K^{2}\big\{ h^{2}
+ \int_{t}^{t_{n+1}} E[|X_{s}-X_{t_{n}}^{N}|^{2}]ds + \int_{t}^{t_{n+1}}E[|Y_{s}-Y_{t_{n}}^{N}|^{2}]ds \nonumber\\
&+&  \int_{t}^{t_{n+1}}E[||Z_{s}-Z_{t_{n}}^{N}||^{2}]ds \big\} + K^{2} h^2 + K^{2}\int_{t}^{t_{n+1}} E[|X_{s}-X_{t_{n+1}}^{N}|^{2}]ds
+ K^{2} \int_{t}^{t_{n+1}}E[|Y_{s}-Y_{t_{n+1}}^{N}|^{2}]ds\nonumber\\
&+& (1+\frac{1}{\epsilon})\alpha^{2}\int_{t}^{t_{n+1}}E[||Z_{s}-\bar{Z}_{t_{n+1}}||^{2}]ds
+ \alpha'\mathds{1}_{\{n<N-1\}}\int_{t_{n+1}}^{t_{n+2}}E[\|\delta Z_s^N\|^2]ds.\nonumber
\end{eqnarray}
Then, by plugging $\bar{Z}_{t_n}$ in the last inequality and from (\ref{intercale-Z-bar}), we obtain

\begin{eqnarray}\label{majAt}
A^n_{t} &\leq & \frac{16K^{2}}{1-\alpha'}\int_{t}^{t_{n+1}} E[|\delta Y_{s}^{N}|^{2}]ds + \frac{1-\alpha'}{4} \big\{ h^{2}
+ \int_{t}^{t_{n+1}} E[|X_{s}-X_{t_{n}}^{N}|^{2}]ds + \int_{t}^{t_{n+1}}E[|Y_{s}-Y_{t_{n}}^{N}|^{2}]ds \nonumber\\
&+& 2\int_{t}^{t_{n+1}}E[||Z_{s}-\bar{Z}_{t_{n}}||^{2}]ds + 2\int_{t_n}^{t_{n+1}}E[||\delta Z_{s}^{N}||^{2}]ds \big\}
+ K^{2} h^2 + K^{2}\int_{t}^{t_{n+1}} E[|X_{s}-X_{t_{n+1}}^{N}|^{2}]\nonumber\\
&+& K^{2} \int_{t}^{t_{n+1}}E[|Y_{s}-Y_{t_{n+1}}^{N}|^{2}]ds
+ (1+\frac{1}{\epsilon})\alpha^{2}\int_{t}^{t_{n+1}}E[||Z_{s}-\bar{Z}_{t_{n+1}}||^{2}]ds \nonumber\\
&+& \alpha'\mathds{1}_{\{n<N-1\}}\int_{t_{n+1}}^{t_{n+2}}E[\|\delta Z_s^N\|^2]ds.\nonumber
\end{eqnarray}
%It is well known from Kloeden and Platen \cite{Kloeden-Platen} that for all $s \in [t_{n},t_{n+1})$ and for all $n=0,\ldots,N-1$
%\begin{eqnarray}\label{estimation X}
%E[|X_{s}-X_{t_{n}}^{N}|^{2}] \leq C h(1+|x|^2)\textrm{ and  } E[|X_{s}-X_{t_{n+1}}^{N}|^{2}] \leq C h(1+|x|^2),
%\end{eqnarray}
%where $C$ is a positive constant independent of $x$ and depending on $K$,$T$, $|b(0)|$ and $\|\sigma(0)\|$.\\
%On the other hand, from Proposition \ref{cvzts}, we have
%$$\displaystyle{E[\Sup_{t_{n}\leq s \leq t_{n+1}}(|Y_{s}-Y_{t_{n+1}}|^{2}+|Y_{s}-Y_{t_{n}}|^{2})] \leq Ch(1+|x|^{2})}.$$
%This implies that
We have
\begin{eqnarray}\label{estiamtion1-Y_s}
E[|Y_{s}-Y_{t_{n+1}}^{N}|^{2}] &\leq& C\{E[|Y_{s}-Y_{t_{n+1}}|^{2}]+E[|\delta Y_{t_{n+1}}^{N}|^{2}] \}\nonumber\\
%&\leq& C\{ h (1+|x|^2) +E[|\delta Y_{t_{n+1}}^{N}|^{2}] \}
\end{eqnarray}
and similarly we have
\begin{eqnarray}\label{estiamtion2-Y_s}
E[|Y_{s}-Y_{t_{n}}^{N}|^{2}] &\leq& C\{ E[|Y_{s}-Y_{t_{n}}|^{2}] +E[|\delta Y_{t_{n}}^{N}|^{2}] \},
\end{eqnarray}
where $C$ is a positive constant independent of $x$.\\
From Lemma \ref{Lemma-estimation X}, (\ref{estiamtion1-Y_s}) and (\ref{estiamtion2-Y_s}), we obtain
\begin{eqnarray}\label{majA^n_t}
A^n_{t}& \leq & C \int_{t}^{t_{n+1}} E[|\delta Y_{s}^{N}|^{2}]ds + C h E[|\delta Y_{t_{n+1}}^{N}|^{2}] + C h E[|\delta Y_{t_{n}}^{N}|^{2}] +  C h^{2}(1+|x|^2)\nonumber\\
&+& C\int_{t}^{t_{n+1}}E[|Y_{s}-{Y}_{t_{n}}|^{2}]ds + C\int_{t}^{t_{n+1}}E[|Y_{s}-{Y}_{t_{n+1}}|^{2}]ds \nonumber\\
&+&  C\int_{t}^{t_{n+1}}E[||Z_{s}-\bar{Z}_{t_{n}}||^{2}]ds  +\frac{1-\alpha'}{2}\int_{t_n}^{t_{n+1}}E[||\delta Z_{s}^{N}||^{2}]ds \nonumber\\
&+& (1+\frac{1}{\epsilon})\alpha^{2}\int_{t}^{t_{n+1}}E[||Z_{s}-\bar{Z}_{t_{n+1}}||^{2}]ds
+\alpha' \mathds{1}_{\{n<N-1\}}\int_{t_{n+1}}^{t_{n+2}}E[\|\delta Z_s^N\|^2]ds.
\end{eqnarray}
where C is a generic positive constant depending on $\alpha'$ and independent of $x$.\\
Using (\ref{majA^n_t}) for  $t=t_{n}$, we get
\begin{eqnarray}\label{majYt}
E[|\delta Y_{t_n}^{N}|^{2}]+\frac{1+\alpha'}{2}\int_{t_n}^{t_{n+1}}E[||\delta Z_{s}^{N}||^{2}]ds
 \leq C \int_{t_n}^{t_{n+1}} E[|\delta Y_{s}^{N}|^{2}]ds+C h E[|\delta Y_{t_{n}}^{N}|^{2}] + B_{n},\textrm{ }
\end{eqnarray}
where we set for all $n=0,\ldots,N-1$:
\begin{eqnarray}\label{}
B_{n}&:=&E[|\delta Y_{t_{n+1}}^{N}|^{2}]+C h E[|\delta Y_{t_{n+1}}^{N}|^{2}]+ C h^{2}(1+|x|^2)\nonumber\\
&+& C\int_{t_n}^{t_{n+1}}E[|Y_{s}-{Y}_{t_{n}}|^{2}]ds + C\int_{t_n}^{t_{n+1}}E[|Y_{s}-{Y}_{t_{n+1}}|^{2}]ds \nonumber\\
&+&  C\int_{t_n}^{t_{n+1}}E[||Z_{s}-\bar{Z}_{t_{n}}||^{2}]ds
\nonumber\\
&+& (1+\frac{1}{\epsilon})\alpha^{2}\int_{t_n}^{t_{n+1}}E[||Z_{s}-\bar{Z}_{t_{n+1}}||^{2}]ds
+\alpha' \mathds{1}_{\{n<N-1\}}\int_{t_{n+1}}^{t_{n+2}}E[\|\delta Z_s^N\|^2]ds.
\end{eqnarray}
From (\ref{majYt}), we obtain
\begin{eqnarray*}
\int_{t_n}^{t_{n+1}}E[||\delta Z_{s}^{N}||^{2}]ds\leq C(h\sup_{t\in[t_n,t_{n+1}]}E[|\delta Y_{t}^{N}|^{2}])+B_n.
%(1-Ch)\sup_{t\in[t_n,t_{n+1}]}E[|\delta Y_{t}^{N}|^{2}]\leq B_n.
\end{eqnarray*}
Combining the previous inequality with (\ref{majA^n_t}), we get for $h$ small enough
\begin{eqnarray*}
\sup_{t\in[t_n,t_{n+1}]}E[|\delta Y_{t}^{N}|^{2}]\leq C B_n,
\end{eqnarray*}
which proves the first part of the Lemma.\\
Inserting the previous inequality into (\ref{majYt}), we get
\begin{eqnarray*}
E\Big[|\delta Y_{t_n}^{N}|^{2}+\frac{1+\alpha'}{2}\int_{t_n}^{t_{n+1}}||\delta Z_{s}^{N}||^{2}ds\Big]
 \leq (1+Ch)\Big\{E\Big[|\delta Y_{t_{n+1}}^{N}|^{2}+\alpha' \mathds{1}_{\{n<N-1\}}\int_{t_{n+1}}^{t_{n+2}}\|\delta Z_s^N\|^2ds \Big]+R_{n}\Big\},
\end{eqnarray*}
which proves the second part of the Lemma.
\ep
\subsection{{\bf Proof of Proposition \ref{propderiv1}.}}
To simplify the notations, we restrict ourselves to the case $k=d=l=1$.
$(D_{\theta}Y,D_{\theta}Z)$ is well defined and from inequalities \reff{apriori1} and \reff{aprioriDX}, we deduce that for each $\theta\leq T$
\begin{equation*}
E[\Sup_{t\leq s\leq T}|D_{\theta}Y_s|^2]+E[\int_t^T|D_{\theta}Z_s|^2ds]\leq C (1+|x|^{2}).
\end{equation*}
We define recursively the sequence $(Y^{m},Z^{m})$ as follows. First we set $(Y^{0},Z^{0})=(0,0)$. Then, given $(Y^{m-1},Z^{m-1})$,
we define $(Y^{m},Z^{m})$ as the unique solution in $\mathbb{S}_k^{2}([t,T])\times \mathbb{H}^2_{k\times d}([t,T])$ of
\begin{eqnarray*}
Y_{s}^{m}=\Phi(X_{T}^{t,x})+\int_{s}^{T}f(r,X_{r}^{t,x},Y_{r}^{m-1},Z_{r}^{m-1})dr
+\int_{s}^{T}g(r,X_{r}^{t,x},Y_{r}^{m-1},Z_{r}^{m-1})\overleftarrow{dB_{r}}
-\int_{s}^{T}Z_{r}^{m}dW_{r}.
\end{eqnarray*}
%From the proof of theorem 1.1 in Pardoux and Peng \cite{par94}, $(Y^{m},Z^{m})$ is a Cauchy sequence
%in $\Bc^2([t,T],\mathbb{D}^{1,2})$ which converges to $(Y,Z)$.\\
We recursively show that $(Y^{m},Z^{m})\in\mathcal{B}^{2}([t,T],\mathbb{D}^{1,2})$.
Suppose that $(Y^{m},Z^{m})\in \Bc^2([t,T],\mathbb{D}^{1,2})$ and let us show that
$(Y^{m+1},Z^{m+1})\in \Bc^2([t,T],\mathbb{D}^{1,2})$.\\
From the induction assumption, we have $\Phi(X_{T})+\int_{s}^{T}f(r,\Sigma_{r}^{m})dr\in\mathbb{D}^{1,2}$.\\
We have $g(r,\Sigma_{r}^{m})\in  \mathbb{D}^{1,2}$ for all $r\in [t,T]$. From Lemma \ref{backD12},
we have $\int_{t}^{T}g(r,\Sigma_{r}^{m})\overleftarrow{dB_{r}}\in \mathbb{D}^{1,2}$.
then
\begin{eqnarray*}
Y_{s}^{m+1}=E\big[ \Phi(X_{T}^{t,x})+\int_{s}^{T}f(r,\Sigma_{r}^{m})dr+\int_{s}^{T}g(r,\Sigma_{r}^{m})\overleftarrow{dB_{r}}|\mathcal{F}_{t,s}^{W}\vee \mathcal{F}_{t,T}^{B}\big]\in\mathbb{D}^{1,2},
\end{eqnarray*}
where $\Sigma_{r}^{m}:=(X_{r}^{t,x},Y_{r}^{m},Z_{r}^{m})$.\\
%Indeed if $Y_{t}\in\mathbb{D}^{1,2}$ then, from the representation theorem martingale
%$E[Y_{t}|\mathcal{F}_{t}]=E[Y_{0}]+\int_{0}^{T}U_{s}dW_{s}$ and so $E[Y_{t}|\mathcal{F}_{t}]\in\mathbb{D}^{1,2}$\\
%From Assumption {\bf (H2)}, Estimations \reff{integrability}-\reff{aprioriDX}-\reff{apriori1} and the induction assumption,
%there exists a positive constant $C$ such that for all $t\in [0,T]$, we have
%\begin{eqnarray*}
%||g(t,\Sigma_{t}^{m})||_{1,2}&=&E[|g(t,\Sigma_{t}^{m})|^2]+\Sum_{i=1}^{k}E[\int_0^T|D_sg_i(t,\Sigma_{t}^{m})|^2ds]\\
%&\leq& K \Big(E[|X_t|^2+|Y_t^m|^2\Big)]+\alpha E[|Z_t^m|^2]\Big)+|g(t,0,0,0)|\\
%&+&\Sum_{i=1}^{k}E[\int_0^T| \nabla_{x}g_i(t,\Sigma_{t}^{m})D_{s}X_{t}
%+\nabla_{y}g_i(t,\Sigma_{t}^{m})D_{s}Y_{t}^{m}\\
%&+&\nabla_{z}g_i(t,\Sigma_{t}^{m})D_{s}Z_{t}^{m} |^2ds]\\
%&\leq &C.
%\end{eqnarray*}
Hence
\begin{eqnarray*}
\int_{t}^{T}Z_{r}^{m+1}dW_{r}=\Phi(X_{T}^{t,x})+\int_{t}^{T}f(r,\Sigma_{r}^{m})dr+
\int_{t}^{T}g(r,\Sigma_{r}^{m})\overleftarrow{dB_{r}}-Y_{t}^{m+1}\in\mathbb{D}^{1,2}.
\end{eqnarray*}
It follows from Lemma \ref{forD12} that $Z^{m+1}\in \mathcal{M}^{2}_{k\times d}([t,T],\mathbb{D}^{1,2})$
and we have $D_{\theta}Y_{s}^{m+1}= D_{\theta}Z_{s}^{m+1}=0$ for  $t\leq s\leq \theta$ and for $\theta\leq s \leq T$, we have
\begin{eqnarray}\label{eqiter}
&&D_{\theta}Y_{s}^{m+1}=\nabla\Phi(X_{T}^{t,x})D_{\theta}X_{T}^{t,x}\\
&&+\int_{s}^{T}\Big( \nabla_{x}f(r,\Sigma_{r}^{m})
 D_{\theta}X_{r}+\nabla_{y}f(r,\Sigma_{r}^{m})D_{\theta}Y_{r}^{m}+\nabla_{z}f(r,\Sigma_{r}^{m})D_{\theta}Z_{r}^{m}\Big)dr\nonumber\\
 &&+\int_{s}^{T}\Big( \nabla_{x}g(r,\Sigma_{r}^{m})
 D_{\theta}X_{r}+ \nabla_{y}g(r,\Sigma_{r}^{m})D_{\theta}Y_{r}^{m}+\nabla_{z}g(r,\Sigma_{r}^{m})
 D_{\theta}Z_{r}^{m}\Big) \overleftarrow{dB_{r}} \nonumber\\
 &&- \int_{s}^{T}D_{\theta}Z_{r}^{m+1}dW_{r}.\nonumber
\end{eqnarray}
From inequality \reff{apriori1}, we deduce that for each $\theta\leq T$
\begin{equation*}\label{aprioriDYm}
E[\Sup_{t\leq s\leq T}|D_{\theta}Y_s^{m+1}|^2]+E[\int_t^T|D_{\theta}Z_s^{m+1}|^2ds]\leq C (1+|x|^{2}).
\end{equation*}
It is known that inequality \reff{apriori1} holds for $(Y^{m+1},Z^{m+1})$ and so we deduce that
\begin{equation*}
\|Y^{m+1}\|_{1,2}+\|Z^{m+1}\|_{1,2}<\infty,
\end{equation*}
which shows that $(Y^{m+1},Z^{m+1})\in \Bc^2([t,T],\mathbb{D}^{1,2})$.
Using the contraction mapping argument as in El Karoui, Peng and Quenez
\cite {kar}, we deduce that  $(Y^{m+1},Z^{m+1})$ converges to
$(Y,Z)$ in $\mathbb{S}^2([t,T])\times \mathbb{H}^2([t,T])$.
We will show that $(D_{\theta}Y^{m},D_{\theta}Z^{m})$ converges to $(Y^{\theta},Z^{\theta})$
 in $L^2(\Omega \times [t,T]\times [t,T] ,dP\otimes dt\otimes dt)$, where $Y_{s}^{\theta}=Z_{s}^{\theta}=0$ for all $t\leq s\leq \theta $
and $(Y_{s}^{\theta},Z_{s}^{\theta},\theta\leq s\leq T)$ is the solution
 of the following BDSDE
\begin{eqnarray}\label{eqlim}
Y_{s}^{\theta}&=&\nabla\Phi(X_{T}^{t,x})D_{\theta}X_{T}^{t,x}\\
&+&\int_{s}^{T}\Big(\nabla_{x}f(r,\Sigma_{r})
 D_{\theta}X_{r}+\nabla_{y}f(r,\Sigma_{r})Y_{r}^{\theta}+\nabla_{z}f(r,\Sigma_{r})
 Z_{r}^{\theta}\Big)dr\nonumber\\
 &+&\int_{s}^{T}\Big( \nabla_{x}g(r,\Sigma_{r})
 D_{\theta}X_{r}+\nabla_{y}g(r,\Sigma_{r})Y_{r}^{\theta}+\nabla_{z}g(r,\Sigma_{r})
 Z_{r}^{\theta}\Big)\overleftarrow{dB_{r}} \nonumber\\
 &-& \int_{s}^{T}Z_{r}^{\theta}dW_{r}.\nonumber
\end{eqnarray}
From equations \reff{eqiter} and \reff{eqlim}, we have
\begin{eqnarray*}
&&D_{\theta}Y_{s}^{m+1}-Y_{s}^{\theta}=
\int_{s}^{T}\Big((\nabla_{x}f(r,\Sigma_{r}^{m})
-\nabla_{x}f(r,\Sigma_{r}))D_{\theta}X_{r}^{t,x}\\
&&+\nabla_{y}f(r,\Sigma_{r}^{m})D_{\theta}Y_{r}^{m}
-\nabla_{y}f(r,\Sigma_{r})Y_{r}^{\theta}
+\nabla_{z}f(r,\Sigma_{r}^{m})D_{\theta}Z_{r}^{m}
-\nabla_{z}f(r,\Sigma_{r})Z_{r}^{\theta}\Big)dr\\
&&+\int_{s}^{T}\Big((\nabla_{x}g(r,\Sigma_{r}^{m})
-\nabla_{x}g(r,\Sigma_{r}))D_{\theta}X_{r}^{t,x}
+\nabla_{y}g(r,\Sigma_{r}^{m})D_{\theta}Y_{r}^{m}
-\nabla_{y}g(r,\Sigma_{r})Y_{r}^{\theta}\Big)\overleftarrow{dB_{r}}\\
&&+\int_{s}^{T}\Big(\nabla_{z}g(r,\Sigma_{r}^{m})D_{\theta}Z_{r}^{m}
-\nabla_{z}g(r,\Sigma_{r})Z_{r}^{\theta}\Big)\overleftarrow{dB_{r}}\\
&&-\int_{s}^{T}(D_{\theta}Z_{r}^{m+1}-Z_{r}^{\theta})dW_{r}.
\end{eqnarray*}
From Proposition \ref{aprioriest}, we have
\begin{eqnarray}\label{m3est}
&&E[\sup_{\theta \leq s\leq T}|D_{\theta}Y_{s}^{m+1}-Y_{s}^{\theta}|^{2}]+E[\int_{s}^{T}|D_{\theta}Z_{r}^{m+1}-Z_{r}^{\theta}|^{2}dr]\\
&&\leq CE\Big[\int_{s}^{T}\Big|(\nabla_{x}f(r,\Sigma_{r}^{m})
-\nabla_{x}f(r,\Sigma_{r}))D_{\theta}X_{r}^{t,x}\nonumber
+\nabla_{y}f(r,\Sigma_{r}^{m})Y_{r}^{\theta}
-\nabla_{y}f(r,\Sigma_{r})Y_{r}^{\theta}\nonumber\\
&&+\nabla_{z}f(r,\Sigma_{r}^{m})Z_{r}^{\theta}
-\nabla_{z}f(r,\Sigma_{r})Z_{r}^{\theta}\Big|^2dr\Big]\nonumber\\
&&+CE\Big[\int_{s}^{T}\Big|(\nabla_{x}g(r,\Sigma_{r}^{m})
-\nabla_{x}g(r,\Sigma_{r}))D_{\theta}X_{r}\nonumber
+\nabla_{y}g(r,\Sigma_{r}^{m})Y_{r}^{\theta}
-\nabla_{y}g(r,\Sigma_{r})Y_{r}^{\theta}\nonumber\\
&&+\nabla_{z}g(r,\Sigma_{r}^{m})Z_{r}^{\theta}
-\nabla_{z}g(r,\Sigma_{r})Z_{r}^{\theta}\Big|^{2}dr\Big].\nonumber
\end{eqnarray}
Therefore, we obtain
\begin{eqnarray}\label{YmYth}
%\begin{split}
&&E[\int_t^T\int_t^T|D_{\theta}Y_s^{m+1}-Y^{\theta}_s|^{2}dsd\theta]
+E[\int_t^T\int_t^T|D_{\theta}Z_s^{m+1}-Z^{\theta}_s|^{2}dsd\theta]\\
&\leq& C E[\int_t^T\int_{t}^{T}|\delta_{r,\theta}^{m}|^2dr d\theta] +C E[\int_t^T\int_{t}^{T}|\rho_{r,\theta}^{m}|^2dr d\theta],\nonumber
%\end{split}
\end{eqnarray}
where
\begin{eqnarray}\label{del}
\delta_{r,\theta}^{m}&=&(\nabla_{x}f(r,\Sigma_{r}^{m})
-\nabla_{x}f(r,\Sigma_{r}))D_{\theta}X_{r}^{t,x}
+\nabla_{y}f(r,\Sigma_{r}^{m})Y_{r}^{\theta}
-\nabla_{y}f(r,\Sigma_{r})Y_{r}^{\theta}\nonumber\\
&+&\nabla_{z}f(r,\Sigma_{r}^{m})Z_{r}^{\theta}
-\nabla_{z}f(r,\Sigma_{r})Z_{r}^{\theta},
\end{eqnarray}
and
\begin{eqnarray}\label{ro}
\rho_{r,\theta}^{m}&=&(\nabla_{x}g(r,\Sigma_{r}^{m})
-\nabla_{x}g(r,\Sigma_{r}))D_{\theta}X_{r}^{t,x}
+\nabla_{y}g(r,\Sigma_{r}^{m})Y_{r}^{\theta}
-\nabla_{y}g(r,\Sigma_{r})Y_{r}^{\theta}\nonumber\\
&+&\nabla_{z}g(r,\Sigma_{r}^{m})Z_{r}^{\theta}
-\nabla_{z}g(r,\Sigma_{r})Z_{r}^{\theta}.
\end{eqnarray}
From the definition of  $(\delta_{r,\theta}^{m})_{t\leq r,\theta\leq T}$, we have
$E[\int_t^T\int_{t}^{T}|\delta_{r,\theta}^{m}|^2dr d\theta]\leq C
\int_t^T(A_{m}(\theta,t,T)+B_{m}(\theta,t,T))d\theta$, where
\begin{eqnarray*}
A_{m}(\theta,t,T)&=&E\Big[\int_{t}^{T}|(\nabla_{x}f(r,\Sigma_{r}^{m})
-\nabla_{x}f(r,\Sigma_{r}))D_{\theta}X_{r}^{t,x}|^2 dr\Big],\\
B_{m}(\theta,t,T)&=&E\Big[\int_{t}^{T}\big|(\nabla_{y}f(r,\Sigma_{r})
-\nabla_{y}f(r,\Sigma_{r}^{m}))Y_{r}^{\theta}\big|^2dr\Big]\\
&+&E\Big[\int_{t}^{T}\big|(\nabla_{z}f(r,\Sigma_{r})
-\nabla_{z}f(r,\Sigma_{r}^{m}))Z_{r}^{\theta}\big|^2dr\Big]
\end{eqnarray*}
Moreover, since $\nabla_{x}f$ is bounded and continuous with respect to $(x,y,z)$,
it follows by the dominated convergence theorem and inequality \reff{integrability}  that
%and since, from Lemma 2.7 \cite{7}, $\int_{0}^{T}\|D_{\theta}X\|^{2}<\infty$ , then
\begin{eqnarray}\label{A}
\Lim_{m\rightarrow \infty}\int_t^TA_{m}(\theta,t,T)d\theta=0.
\end{eqnarray}
Furthermore, since $\nabla_{y}f$ and $\nabla_{z}f$ are bounded and continuous with respect to $(x,y,z)$,
%and since, from Theorem 2.9 \cite{7}
%$\int_{0}^{T}\|Y^{\theta}\|_{S^{2}}^{2}+\|Z^{\theta}\|_{M^{2}}^{2}d\theta<\infty$,
it follows, also, by the dominated convergence theorem and inequality \reff{apriori1} that
\begin{eqnarray}\label{B}
\Lim_{m\rightarrow \infty}\int_t^TB_{m}(\theta,t,T)d\theta=0.
\end{eqnarray}
From the definition of  $(\rho_{r,\theta}^{m})_{s\leq r,\theta\leq T}$, we have
$E[\int_t^T\int_{t}^{T}|\rho_{r,\theta}^{m}|^2dr d\theta]\leq C \int_t^T(A'_{m}(\theta,t,T)+B'_{m}(\theta,t,T))d\theta$,
with
\begin{eqnarray*}\label{ABp}
A'_{m}(\theta,t,T)&=&E\Big[\int_{t}^{T}|(\nabla_{x}g(r,\Sigma_{r}^{m})
-\nabla_{x}g(r,\Sigma_{r}))D_{\theta}X_{r}^{t,x}|^2 dr\Big],\\
B'_{m}(\theta,t,T)&=&E\Big[\int_{t}^{T}\big|(\nabla_{y}g(r,\Sigma_{r})
-\nabla_{y}g(r,\Sigma_{r}^{m}))Y_{r}^{\theta}\big|^2dr\big]\\
&+&E\Big[\int_{t}^{T}\big|(\nabla_{z}g(r,\Sigma_{r})
-\nabla_{z}g(r,\Sigma_{r}^{m}))Z_{r}^{\theta}\big|^2dr\Big].\nonumber
\end{eqnarray*}
Similarly as shown above, since $\nabla_{y}g$ and $\nabla_{z}g$ are bounded and continuous with respect to $(x,y,z)$ we can
show that:
\begin{eqnarray}\label{ABp}
\Lim_{m\rightarrow \infty}\int_t^TA'_{m}(\theta,t,T)d\theta
=\Lim_{m\rightarrow \infty}\int_t^T B'_{m}(\theta,t,T) d\theta=0.
\end{eqnarray}
Plugging \reff{A}, \reff{B} and \reff{ABp} into inequality \reff{YmYth}, we deduce that
\begin{eqnarray*}
\Lim_{m\rightarrow \infty}E[\int_t^T\int_t^T|D_{\theta}Y_s^{m+1}-Y^{\theta}_s|^{2}dsd\theta]
+E[\int_t^T\int_t^T|D_{\theta}Z_s^{m+1}-Z^{\theta}_s|^{2}dsd\theta]=0.
\end{eqnarray*}
It follows that $(Y^{m},Z^{m})$ converges to $(Y,Z)$ in $L^2([t,T],\mathbb{D}^{1,2}\times \mathbb{D}^{1,2})$
and a version of $(D_{\theta}Y,D_{\theta}Z)$ is given by $(Y^{\theta},Z^{\theta})$, which is the desired result.
\ep
%%%%%%%%%%%%%%%%%%%%%%%%%%%%%%%%%%%%%% Proof of Proposition 4.3 %%%%%%%%%%%%%%%%%%%%%%%%%%%%%%%%%%%%%%%%%
\subsection{ Second order Malliavin derivative of  the solution of  BDSDE's}
We apply similar computation to get the second order Malliavin  derivative  representations of the solution of BDSDE 's, so we will omit the proof.
\begin{Proposition}\label{propderiv2}
We set $t\in [0,T]$. Then, under Assumptions {\bf (H2)} and {\bf(H3)}, for each $t\leq \theta\leq T$,
$(D_{\theta}Y,D_{\theta}Z)$ belongs to $\mathcal{B}^{2}([t,T],\mathbb{D}^{1,2})$.
For each $t\leq v \leq T$ and $1\leq i,j\leq d$, we have
\begin{eqnarray*}
D_{v}^jD_{\theta}^iY_{s}=D_{v}^jD_{\theta}^iZ_{s}^{n}=0,\,\,1\leq n\leq d, \mbox{ if }s < \theta\vee v,
\end{eqnarray*}
and a version of $(D_{v}^jD_{\theta}^iY_{s},D_{v}^jD_{\theta}^iZ_{s})_{ v\vee \theta \leq s\leq T}$
is the unique solution of the following equation:
\begin{eqnarray*}
D_{v}^jD_{\theta}^iY_{s}=T_1(\Phi)+T_2(f)+T_3(g)+T_4(W),
\end{eqnarray*}
where
\begin{eqnarray*}
T_1(\Phi)=\Sum_{n_1=1}^{k}\nabla((\nabla \Phi)^{n_1}(X_{T}^{t,x}))D_{v}^jX_{T}^{t,x}(D_{\theta}^iX_{T}^{t,x})^{n_1}
+\nabla\Phi(X_{T}^{t,x})D_{v}^jD_{\theta}^iX_{T}^{t,x},
\end{eqnarray*}
\begin{eqnarray*}
T_2(f)&=&\int_{s}^{T}\Sum_{n_1=1}^{k}\Big(\nabla_x((\nabla_x f)^{n_1}(r,X_{r}^{t,x},Y_{r},Z_{r}))D_{v}^jX_{r}^{t,x}
(D_{\theta}^iX_{r}^{t,x})^{n_1}\nonumber\\
&+& \nabla_x f(r,X_{r}^{t,x},Y_{r},Z_{r})D_{v}^jD_{\theta}^iX_{r}^{t,x}\Big) dr\nonumber\\
&+&\int_{s}^{T}\Big(\Sum_{n_1=1}^{k}\nabla_y((\nabla_y f)^{n_1}(r,X_{r}^{t,x},Y_{r},Z_{r}))D_{v}^jY_{r}
(D_{\theta}^iY_{r})^{n_1}\nonumber\\
&+&\nabla_y f(r,X_{r}^{t,x},Y_{r},Z_{r})D_{v}^jD_{\theta}^iY_{r}\Big)dr \nonumber\\
&+&\Sum_{n_2=1}^{d}\int_{s}^{T}\Sum_{n_1=1}^{k}\nabla_{z^{n_2}}((\nabla_{z^{n_2}} f)^{n_1}(r,X_{r}^{t,x},Y_{r},Z_{r}))D_{v}^jZ_{r}^{n_2}
(D_{\theta}^iZ_{r}^{n_2})^{n_1}dr\nonumber\\
&+& \Sum_{n_2=1}^{d}\int_{s}^{T}\nabla_{z^{n_2}} f(r,X_{r}^{t,x},Y_{r},Z_{r})D_{v}^jD_{\theta}^iZ_{r}^{n_2}  dr,\nonumber\\
\end{eqnarray*}
\begin{eqnarray*}
T_3(g)&=&\Sum_{n_3=1}^{l}\int_{s}^{T}\Sum_{n_1=1}^{k}\nabla_x((\nabla_x g^{n_3})^{n_1}(r,X_{r}^{t,x},Y_{r},Z_{r}))D_{v}^jX_{r}^{t,x}
(D_{\theta}^iX_{r}^{t,x})^{n_1}\overleftarrow{dB_{r}^{n_3}}\nonumber\\
&+&\Sum_{n_3=1}^{l}\int_{s}^{T}\nabla_x g^{n_3}(r,X_{r}^{t,x},Y_{r},Z_{r})D_{v}^jD_{\theta}^iX_{r}^{t,x} \overleftarrow{dB_{r}^{n_3}}\nonumber\\
&+&\Sum_{n_3=1}^{l} \int_{s}^{T}\Sum_{n_1=1}^{k}\nabla_y((\nabla_y g^{n_3})^{n_1}(r,X_{r}^{t,x},Y_{r},Z_{r}))D_{v}^jY_{r}
(D_{\theta}^iY_{r})^{n_1}\overleftarrow{dB_{r}^{n_3}}\nonumber\\
&+&\Sum_{n_3=1}^{l} \int_{s}^{T}\nabla_y g^{n_3}(r,X_{r}^{t,x},Y_{r},Z_{r})D_{v}^jD_{\theta}^iY_{r}
\overleftarrow{dB_{r}^{n_3}} \nonumber\\
&+&\Sum_{n_3=1}^{l}\Sum_{n_2=1}^{d}\int_{s}^{T}\Sum_{n_1=1}^{k}\nabla_{z^{n_2}}((\nabla_{z^{n_2}} g^{n_3})^{n_1}(r,X_{r}^{t,x},Y_{r},Z_{r}))D_{v}^jZ_{r}^{n_2}
(D_{\theta}^iZ_{r}^{n_2})^{n_1}\overleftarrow{dB_{r}^{n_3}}\nonumber\\
&+& \Sum_{n_3=1}^{l}\Sum_{n_2=1}^{d}\int_{s}^{T}\nabla_{z^{n_2}} g^{n_3}(r,X_{r}^{t,x},Y_{r},Z_{r})D_{v}^jD_{\theta}^iZ_{r}^{n_2} \overleftarrow{dB_{r}^{n_3}},
\end{eqnarray*}
\begin{eqnarray*}
T_4(W)=-\Sum_{n_2=1}^{d}\int_{s}^{T}D_{v}^jD_{\theta}^{i}Z_{r}^{n_2}dW_{r}^{n_2}\nonumber,
\end{eqnarray*}
$(z^j)_{1\leq j\leq d}$ denotes the j-th column of the matrix $z$,
$(g^{n_3})_{1\leq n_3\leq l}$ denotes the $n_3$-th column of the matrix $g$, $B=(B^1,\ldots, B^l)$, $(D_{\theta}^iX_{r}^{t,x})^{n_1}$ is
the $n_1$-th component of the vector $(D_{\theta}^iX_{r}^{t,x})$,
$(D_{\theta}^iY_{r})^{n_1}$ is the $n_1$-th component of the vector $(D_{\theta}^iY_{r})$
and $(D_{\theta}^iZ_{r}^{n_2})^{n_1}$ is the $n_1$-th component of the vector $(D_{\theta}^iZ_{r}^{n_2})$.
\end{Proposition}
\end{document}